\documentclass[leqno,11pt]{article}
\usepackage{amsmath}
\usepackage{amsfonts}
\usepackage{array}
\usepackage{enumerate}
\usepackage{graphicx}
\usepackage{amssymb}
\usepackage{amsthm}
\usepackage{xcolor}

\newtheorem{proposition}{Proposition}
\newtheorem{lemma}{Lemma}
\newtheorem*{corollary}{Corollary}
\newtheorem{definition}{Definition}

\newtheorem{remark}{Remark}
\newtheorem*{example}{Example}
\newcommand{\R}{\mathbb{R}}

\begin{document}

\title{Classical quasi-steady state reduction -- \\
A mathematical characterization}

\author{ Alexandra Goeke\\
Mathematik A, RWTH Aachen \\
52056 Aachen, Germany\\
    \\
Sebastian Walcher\footnote{Corresponding author. Email {\tt walcher@matha.rwth-aachen.de}, Phone +49 241 809 8132, Fax +49 241 809 2212.}\\
Mathematik A, RWTH Aachen \\
52056 Aachen, Germany\\
  \\
Eva Zerz\\
Mathematik D, RWTH Aachen \\
52056 Aachen, Germany\\
}


\maketitle
\begin{abstract} 
We discuss parameter dependent polynomial ordinary differential equations that model chemical reaction networks. By {\em classical quasi-steady state (QSS) reduction} we understand the following familiar (heuristically motivated) mathematical procedure: Set the rate of change for certain (a priori chosen) variables equal to zero and use the resulting algebraic equations to obtain a system of smaller dimension for the remaining variables. This procedure will generally be valid only for certain parameter ranges. We start by showing that the reduction is accurate if and only if the corresponding parameter is  what we call a QSS parameter value, and that the reduction is approximately accurate if and only if the corresponding parameter is close to a QSS parameter value.  The QSS parameter values can be characterized by polynomial equations and inequations, hence parameter ranges for which QSS reduction is valid are accessible in an algorithmic manner. A defining characteristic of a QSS parameter value is that the algebraic variety defined by the QSS relations is invariant for the differential equation. A closer investigation of the associated systems shows the existence of further invariant sets; here singular perturbations enter the picture in a natural manner. We compare QSS reduction and singular perturbation reduction, and show that, while they do not agree in general, they do, up to lowest order in a small parameter, for a quite large and relevant class of examples. This observation, in turn, allows the computation of QSS reductions even in cases where an explicit resolution of the polynomial equations is not possible.

\end{abstract}

\section{Introduction and overview}
In chemical reaction networks it is often observed or assumed that, during a relevant time period, the concentration of certain reactants changes negligibly compared to the overall rate of reaction. This {\em quasi-steady state (QSS)} behavior gives rise to  a heuristic reduction procedure for the ordinary differential equation system governing the reaction network: Idealizing the QSS assumption, one sets the net rate of change for each QSS species (i.e., the corresponding entry on the right-hand side of the differential equation) equal to zero and uses the ensuing algebraic equations to obtain a reduced differential equation of smaller dimension. This procedure, which we call {\em classical QSS reduction}, has proven very useful -- and correct -- in various settings for more than a century. The best known example probably is the  Michaelis-Menten system for the action of an enzyme.\\
From a mathematical perspective, a justification of the heuristics -- and even prior to that, a transfer of the underlying scientific assumption to mathematical terms -- is not obvious. Following several decades of ad hoc arguments, mathematicians in the 1960s started to view QSS as a singular perturbation phenomenon, and the first rigorous convergence proofs were given. Moreover, ``slow-fast" timescale arguments inspired by singular perturbation theory were employed to identify parameter ranges for which QSS holds, and this led to mathematical interpretations of QSS that are based on timescale arguments.\\

It could be said that we start  the present paper by turning back the clock: Our vantage point is to focus on the classical reduction procedure in its own right and to determine under which conditions it is valid. We emphasize that a priori we make no additional assumptions concerning slow/fast dynamics, and we will not a priori assume a singular perturbation setting.\\
Throughout we consider a spatially homogeneous setting with constant thermodynamical parameters, and mass action kinetics. Thus the objects of investigation are parameter dependent polynomial (or rational) ordinary differential equations. QSS is understood here to hold for certain chemical species (i.e. variables); slow and fast reactions (and the related partial equilibrium assumption) will not be discussed.\\
We first review the classical reduction procedure and discuss what is necessary and sufficient for this procedure to work. There are some obvious technical prerequisties to ensure a local resolution of the algebraic equations which are implied by the QSS assumption, and which in turn define the QSS variety as their common zero set.  More importantly, the relevant solutions of the reduced differential equation should approximate the solutions of the original system. This is, in our view, the minimal requirement for any sensible QSS reduction. In turn, this minimal requirement provides nontrivial conditions on parameters (rate constants and initial concentrations). If one requires furthermore that the approximation error should become arbitrarily small then one arrives naturally at the notion of a {\em QSS parameter value}: Solutions of the original system and of the QSS-reduced system near the QSS variety are close (on compact time intervals) if and only if the parameter vector is close to some QSS parameter value. Note that we invoked only a minimal requirement, thus a QSS variety may not be attractive; the behavior of the system near a QSS parameter value may require further analysis. \\
For polynomial (or rational) parameter dependent systems, QSS parameter values can be characterized by algebraic equations and inequations, and therefore they are (in principle) computable. Even more, there is a method to compute QSS parameter values via algorithmic algebra. For the relatively low-dimensional systems under consideration in the present paper, standard algorithms and implementations are sufficient, but higher dimensions (or a larger number of parameters) would require more efficient and specialized methods.\\
At a QSS parameter value the differential equation system admits a distinguished invariant set, viz., the QSS variety. Moreover, this variety is frequently the union of subvarieties of smaller dimension. This observation may explain the prevalence of singular perturbation scenarios when QSS holds, and it also implies that certain {\em affine coordinate subspaces } (with all QSS species having a fixed value) are of particular relevance.\\
We proceed to address a problem which is on the one hand obvious but on the other hand is frequently suppressed: The algebraic obstacles when actually carrying out a classical QSS reduction may be formidable. In particular there are many systems for which an explicit reduction (involving only algebraic operations and radicals) is not feasible or does not even exist. We show that, in spite of this fact, for many relevant settings and appropriate QSS parameter values (corresponding to affine coordinate subspaces) the reduction can be carried out explicitly anyway (up to first order in a suitable ``small parameter"). Thus, while the algebraic problem does not vanish, one can frequently circumnavigate it.\\
In the final section of the paper we discuss examples and applications. Several notions, auxiliary results and supplementary material are collected in the Appendix.
\section{Remarks on classical QSS reduction}
\subsection{Some history}
We sketch the origin and some crucial developments of QSS, and briefly mention some recent work of relevance. It seems that QSS arguments originated with the work of Henri \cite{henri} and Michaelis/Menten \cite{MM}; their heuristic arguments seem to be based on consideration of slow and fast reactions. Briggs and Haldane \cite{briggshaldane} seem to have been the first to write down the familiar QSS reduction for complex in the Michaelis-Menten system (under the assumption of small initial enzyme concentration), which is still an indispensable part of every introductory monograph on physical chemistry or biochemistry (see e.g. Atkins and de Paula \cite{ap}). With the emergence of singular perturbation theory, a natural mathematical framework for QSS and QSS reduction became available; see e.g. Heineken et al. \cite{hta}. The broader framework of computational singular perturbation (CSP) methods was later introduced by Lam and Goussis \cite{lamgo}. In order to justify the reduction procedure for Michaelis-Menten mathematically and, at the same time, to determine parameter regions for which it is applicable, two lines of approach were taken: Schauer and Heinrich \cite{hs-inv} required that the relevant trajectories of the full system remain close to the QSS variety for the Michaelis-Menten system, which is defined by stationary complex concentration); this argument was modified and continued in \cite{nw09} and in \cite{cgw}. The second (more prevalent) approach is due to Segel and Slemrod \cite{SSl} who worked with time scale estimates inspired by singular perturbation theory. Among the numerous follow-up publications to \cite{SSl} we only mention some recent papers, viz. the extensive discussion by Goussis \cite{gouQSS}, a definition of QSS in Kollar and Siskova \cite{kosi} which includes exponential attraction to some manifold, and the work by Radulescu et al. \cite{rvg}, Samal et al. \cite{sgfr}, Samal et al. \cite{sgfwr}  who formalized the slow-fast arguments by employing methods from tropical geometry. The approach by Segel and Slemrod (as well as the publications based on it) requires an a priori designation of ``slow'' and ``fast" variables. In \cite{gwz} a method is presented to determine all parameter values for which singular perturbation reduction in the sense of Tikhonov and Fenichel works, with no a priori assumptions necessary; see Appendix, \ref{sptredu}. Samal et al. \cite{sgfwr} -- roughly speaking -- look for cancellation of fast reaction terms; their a priori assumption is the existence of slow and fast species, but no a priori choice of species is required. 
A practical problem regarding classical QSS reduction is due to the fact that an explicit resolution of the equations stemming from QSS may be cumbersome or even not possible. Using Abel's theorem on the solvability of polynomials by radicals, Pantea et al. \cite{pgrc} recently gave several examples for which an explicit resolution is impossible. At first glance this imposes a serious restriction on the applicability of the method.

\subsection{Benchmark example: Michaelis-Menten}
We refer to a very well-known reaction network and its associated differential equation to review the standard quasi-steady state reductions, and the underlying assumptions. In the course of the paper we will also employ this system for examples and to illustrate some concepts. We also will provide some new aspects for this system in the following sections; in particular we will find all parameter values near which QSS reduction is approximately accurate.\\
The reversible Michaelis-Menten reaction is defined by the reaction scheme
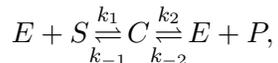
\[
E+S{\overset{k_1}{\underset{k_{-1}}\rightleftharpoons}}
C{\overset{k_2}{\underset{k_{-2}}\rightleftharpoons}} E+P,
\]
with an associated differential equation for the concentrations 
\begin{equation}\label{mimerev}
\begin{array}{clccl}
\dot{s}&=-&k_1e_0s&+&(k_1s+k_{-1})c,\\
\dot{c}&= &k_1e_0s&-&(k_1s+k_{-1}+k_2)c+k_{-2}(e_0-c)(s_0-s-c),\\
\end{array}
\end{equation}
usually with initial values $s(0)=s_0>0$ and $c(0)=0$. In the special case $k_{-2}=0$ one speaks of the irreversible Michaelis-Menten system; with differential equation
\begin{equation}\label{mimeirrev}
\begin{array}{clccl}
\dot{s}&=-&k_1e_0s&+&(k_1s+k_{-1})c,\\
\dot{c}&= &k_1e_0s&-&(k_1s+k_{-1}+k_2)c.\\
\end{array}
\end{equation}

\subsubsection{Quasi-steady state for complex}
Classical quasi-steady state reduction for complex goes back to Briggs and Haldane \cite{briggshaldane}: One assumes that the rate of change for complex concentration is (almost) equal to zero and uses the ensuing algebraic equation to eliminate $c$ from the differential equation for $s$. The familiar result for the irreversible system is the Michaelis-Menten equation
\begin{equation}\label{mimeirrevqssc}
\dot s =-e_0\frac{k_1k_2s}{k_1s+k_{-1}+k_2}. 
\end{equation}
For the reversible system the condition ``$\dot c=0$" yields the quadratic equation
\[
k_1e_0s-(k_1s+k_{-1}+k_2)c+k_{-2}(e_0-c)(s_0-s-c)=0
\]
for $c$, with solution 
\[
c=\frac{1}{2k_{-2}}\left(t-\sqrt{t^2-4e_0k_{-2}(k_1s+k_{-2}(s_0-s))}\right)
\]
(the negative sign is forced by $c\leq e_0$), where
\[
t:=k_1s+k_{-1}+k_2+k_{-2}(e_0+s_0-s).
\]
 One then has to substitute this value for $c$ in the differential equation for $s$. The procedure has been carried out (see e.g. Miller and Alberty \cite{MiAl}) but it is rarely used; one reason may be the unwieldiness of the algebraic manipulations.\\
Implicit in such a procedure is the understanding that it will be valid only in certain parameter regions. A typical assumption for Michaelis-Menten is small initial enzyme concentration; in other words the system is being considered in the limit $e_0\to 0$. But for $e_0\to 0$ there also exists a singular perturbation reduction (on the asymptotic slow manifold defined by $c=0$; see e.g. \cite{nw11}, subsection 3.1), which yields
the reduced equation
\begin{equation}\label{mimerevqssc}
\dot s =-e_0\frac{k_1k_2s+k_{-1}k_{-2}(s-s_0)}{k_1s+k_{-1}+k_2+k_{-2}(s_0-s)}. 
\end{equation}
This coincides with the QSS reduction in the irreversible setting (when $k_{-2}=0$) but has a markedly different appearance from the classical reduction when $k_{-2}\not=0$. However, the right hand side of the equation obtained by QSS reduction and the right hand side of \eqref{mimerevqssc} agree up to first order in the small parameter $e_0$. To verify this, note that for $e_0\ll 1$ one has
\[
\begin{array}{rcl}
c&=&\frac{t}{2k_{-2}}\left(1-\sqrt{1-4e_0k_{-2}(k_1s+k_{-2}(s_0-s))/t^2}\right)\\
 &\approx& \frac{t}{2k_{-2}}\left(e_0k_{-2}(k_1s+k_{-2}(s_0-s))/t^2\right)
\end{array}
\]
with the familiar approximation $\sqrt{1+x}\approx 1+x/2$. Upon substituting this expression in \eqref{mimeirrev}, and considering only terms of lowest order in $e_0$, one obtains \eqref{mimerevqssc}. Thus, although starting from different vantage points, both reduction methods (essentially) yield the same result.
In Section \ref{secstructure} we will provide an explanation for this observation.
\subsubsection{Quasi-steady state for substrate}\label{subsecsubs}
Setting $\dot s=0$ in the reversible system \eqref{mimerev} one finds the classical reduced equation
\[
\dot c=-\frac{k_1k_2+k_{-1}k_{-2}}{k_1}\cdot c +k_{-2}\cdot(e_0-c)(s_0-c)
\]
for quasi-steady state with respect to $s$.
From \cite{gwz2}, Subsection 7.2 (using arguments similar to those in Section \ref{Formalsec} below) one finds that the quasi-steady state reduction for substrate works -- in the very basic sense that the reduced equation provides an approximately accurate solution of \eqref{mimerev} -- whenever $k_{-1}$ is small compared to other parameters. (One may directly infer this from the obvious exact invariance of the line $s=0$ in case $k_{-1}=0$.) \\
For the irreversible system Segel and Slemrod \cite{SSl} used time scale estimates for \eqref{mimeirrev} to predict QSS for substrate (``reverse QSS") when
\[
k_{-1}\approx k_2 \text{  and  }\frac{k_{-1}}{k_1e_0}\ll 1.
\]
These stronger conditions (given that $k_{-2}=0$) may be  translated to $k_{-1}=\epsilon k_{-1}^*$ and $k_2=\epsilon k_2^*$ with $\epsilon \to 0$, or alternatively to $e_0\to\infty$ (with a change of time scale). They lead to a singular perturbation reduction (see \cite{gw}, Example 8.6 and \cite{nw11}, Subsection 3.2 respectively) which is consistent with the QSS reduction.\\
The discussion and a numerical example in \cite{GSWZ1}, Section 4 for the limiting case $k_{-1}\to 0$ (with no condition on $k_2$) show that the approximation quality by the QSS reduction depends on the eigenvalues of the linearization at the stationary point. The curve given by $s=0$ is always invariant, but nearby solutions may not be locally attracted to this curve, since the generic direction of approach to the stationary point $0$ (which is an attracting node) may {\em not} be tangent to $s=0$ but to the other eigenspace of the linearization. (This occurs whenever  $k_2>k_1e_0$.) In such a situation, numerical examples show poor approximation quality for the QSS reduction.\\

There are some further notions of quasi-steady state for Michaelis-Menten, e.g.~the notion of {\em total quasi-steady state (tQSS)} introduced by Borghans et al. \cite{borg}, which we will not discuss in the present paper. 
\section{Classical QSS for chemical species}\label{Formalsec}
We first establish a formal framework for classical QSS reduction of parameter-dependent (reaction) equations. 
\subsection{Notation}
Throughout the paper we will consider an ordinary differential equation
\begin{equation}\label{sys}
\dot x = h(x,\pi),\quad x\in \mathbb R^n,\quad \pi\in \mathbb R^m
\end{equation}
with $h$ a polynomial in variables $x$ and parameters $\pi$. (Most results also hold for, or are readily adapted to, rational functions.)
We think of this system as describing the time evolution of a spatially homogeneous chemical reaction network with mass-action kinetics and fixed thermodynamical parameters. Therefore we are mostly interested in settings when all the parameters, which represent rate constants or initial concentrations, are nonnegative, and for every nonnegative parameter vector the positive orthant $\mathbb R_+^n$ is positively invariant for \eqref{sys}. (The variables represent concentrations of chemical species.) 

As a matter of notation, by $Dh(x,\,\pi)$ we denote the derivative of $h$ with respect to $x$. For any smooth function $\theta:\, V\to \mathbb R$ (with $V$ an open subset of $\mathbb R^n\times \mathbb R^m$), we denote by  $L_h(\theta)$ the Lie derivative with respect to $x$, i.e.
\[
L_h(\theta)(x,\pi)=D\theta(x,\pi)h(x,\pi).
\]
Lie derivatives play an important role in invariance criteria; see Lemma \ref{invlem} in the Appendix, \ref{invapp}.
\subsection{The QSS reduction procedure}\label{subs32}
The basic procedure underlying the classical reduction heuristics is to eliminate certain variables by setting their rates of change equal to zero, and to utilize the resulting algebraic equations. In many familiar instances of QSS reduction, the algebraic equations are amenable to an explicit solution.  But this is not always the case and should even be considered an exception. Therefore we introduce, in addition to explicit QSS reduction, also implicit QSS reduction. In the latter scenario the reduced equation ``lives'' on an algebraic subvariety of $\mathbb R^n$. Up to coordinate transformations (which may not be explicitly available) the two versions are equivalent.   \\

 In the following, let $1\leq r<n$; we will consider QSS reduction of \eqref{sys} with respect to the ``species" $x_{r+1},\ldots, x_n$. We fix some notation.
\begin{itemize}
\item Let $1\leq r<n$ and
\[
\begin{array}{rcl}
x^{[1]}&:=&(x_1,\ldots, x_r)^{\rm tr}; \quad x^{[2]}:=(x_{r+1},\ldots, x_n)^{\rm tr}\\
h^{[1]}&:=&(h_1,\ldots, h_r)^{\rm tr}; \quad h^{[2]}:=(h_{r+1},\ldots, h_n)^{\rm tr}.
\end{array}
\]
By $D_i$ we denote the partial derivative with respect to $x^{[i]}$.
\item Given $\pi\in \mathbb R^m$, we let $Y_\pi$ be the set of zeros of $h^{[2]}(\cdot,\pi)$. (This is an algebraic variety.)
\end{itemize}
\begin{definition}\label{qssbase} 
If there is $y\in Y_\pi$ such that $D_2h^{[2]}$ has full rank $n-r$ at $(y,\pi)$ then we denote by $U_\pi\subseteq Y_\pi$ a relatively Zariski-open neighborhood of  $y$ in which this rank is maximal. We will furthermore assume (with no loss of generality) that $U_\pi$ is irreducible, and call $U_\pi$ a {\em QSS variety} with respect to $x_{r+1},\ldots, x_n$.
\end{definition}
In this definition we relied on some elementary properties of algebraic varierties, which are recalled in the Appendix, \ref{algvarapp}. The rank condition in Definition \ref{qssbase} ensures that $U_\pi$ is a submanifold of dimension $r$. Moreover, by the implicit function theorem, there exists a smooth function $\Psi$ of $x^{[1]}$ (defined on some open set in $\mathbb R^r$) such that a neighborhood $\widetilde U_\pi\subseteq U_\pi$ of $(y,\pi)$ can be represented as the graph of $\Psi$. The following provides a description of the classical reduced equation.
\begin{definition}\label{qssexpl} Assume that the rank condition for  $D_2h^{[2]}$ from Definition \ref{qssbase} holds at $(y,\pi)$, and let $\widetilde U_\pi\subset U_\pi$ be the graph of the smooth function $\Psi$. Then the differential equation
\begin{equation}\label{qssredsysex}
\dot x^{[1]}= h^{[1]}( x^{[1]},\Psi( x^{[1]}),\pi)
\end{equation}
will be called an {\em explicit QSS reduction} of \eqref{sys} near $(y,\pi)$, with respect to the species $x_{r+1},\ldots, x_n$.
\end{definition}
An explicit form (e.g. involving only radicals) of $\Psi$ may not exist; see Pantea et al. \cite{pgrc}. This is one reason to introduce a second version.
\begin{definition}\label{QSSdef} Let the notation and assumptions of Definition \ref{qssbase} be given. Then
 the following equation will be called an {\em implicit QSS-reduced equation} of \eqref{sys} on $U_\pi$, with respect to the species $x_{r+1},\ldots, x_n$:
\begin{equation}\label{qssredsys}
\begin{array}{rcl}
\dot x^{[1]}&=& h^{[1]}(x,\pi)\\
\dot x^{[2]}&=& -D_2h^{[2]}(x,\pi)^{-1}D_1h^{[2]}(x,\pi)h^{[1]}(x,\pi).
\end{array}
\end{equation}
We will briefly write $\dot x =h_{\rm red}(x,\pi)$ for this equation.
\end{definition}
These two versions admit (locally) the same solutions, in the following sense.
\begin{lemma}\label{eximplem}
\begin{enumerate}[(a)]
\item Given the setting of Definition \ref{QSSdef}, the variety $U_\pi$ is invariant for system \eqref{qssredsys}.
\item For any solution $z(t)=(z^{[1]}(t),\, z^{[2]}(t))$ of \eqref{qssredsys} on $U_\pi$ one has that $z^{[1]}(t)$ locally solves system \eqref{qssredsysex}. For any solution $v(t)$ of \eqref{qssredsysex} one has that $(v(t),\,\Psi(v(t))$ locally solves system  \eqref{qssredsys}.
\end{enumerate}
\end{lemma}
\begin{proof} To prove invariance, verify that $L_{h_{\rm red}}(h_j)=0$ for $r+1\leq j\leq n$ and use Lemma \ref{invlem} in the Appendix, \ref{invapp}. Part (b) follows by invariance and differentiation rules.
\end{proof}
The reasoning which underlies Definition \ref{QSSdef} and Lemma \ref{eximplem} is known from the literature; see Gear and Kevrekidis \cite{GeKe}, Zagaris et al. \cite{ZKK}. In particular, equation \eqref{qssredsys} can be derived from \cite{ZKK}, equation (3.5) with $f^s$ standing for the rates of change of the QSS variables, and also from \cite{ZKK}, equation (5.8) assuming the iteration is stationary.\\
It may be advantageous to employ an implicit version of the reduction; see Bennett et al. \cite{bvth}, Kumar and Josic \cite{kujo}, Section 2. In any case, the implicit version will prove useful for discussing questions of accuracy. Essentially the same characterization of a reduced system is used by Kollar and Siskova \cite{kosi} in their definition and analysis of QSS reduction. We note that there exist different-looking versions of \eqref{qssredsys} on the variety; see Appendix, \ref{varredsys}.
\subsection{Accuracy and approximate accuracy}
So far we only discussed the formalities of the QSS reduction procedure but we were not concerned with any actual correspondence between solutions of \eqref{sys} and \eqref{qssredsys}. Indeed there is no a priori reason to assume any such correspondence, and this is the focus of the present subsection. If a parameter value is such that the QSS variety is invariant for \eqref{sys} then we call it a {\em QSS parameter value}. We will instantly show that the QSS reduction is accurate (i.e., solutions of \eqref{sys} and \eqref{qssredsys} with initial values on the QSS variety are equal) if and only if one has a QSS parameter value. By continuous dependence one obtains that the QSS reduction is approximately accurate if a parameter value is close to a QSS parameter value. Using a further (elementary but possibly less familiar) argument we will show that the QSS reduction is approximately accurate (up to an arbitrarily small error) only if the parameter is close to a QSS parameter value. We carry this out in detail because it is a crucial point: We obtain a minimal requirement for validity of the QSS reduction procedure at some given parameter.
\begin{definition}\label{qsspvdef}
We call the parameter value $\pi^*$ {\em a QSS parameter value} with respect to the species $x_{r+1},\ldots, x_n$ if $D_2h^{[2]}(y,\pi^*)$ has rank $n-r$ at some $y\in Y_{\pi^*}$, and $U_{\pi^*}$ is invariant for \eqref{sys}. 
\end{definition}
By irreducibility, this is equivalent to invariance of the intersection of $U_{\pi^*}$ with some neighborhood of $y^*\in U_{\pi^*}$.
We first show that at QSS parameter values, and only at these, the reduction provides solutions of the original system \eqref{sys}.
\begin{proposition}\label{consexact}
Let $\pi$ be given such that the rank condition on $D_2h^{[2]}$ from Definition \ref{qssbase} holds, and let $(y,\pi)\in U_\pi$. Then the following are equivalent.
\begin{enumerate}[(a)]
\item The solutions of \eqref{sys} and of  \eqref{qssredsys} with initial value in $U_\pi$ are equal.
\item $U_\pi$ is  invariant with respect to \eqref{sys}.
\end{enumerate}
\end{proposition}
\begin{proof} According to Lemma \ref{invlem} invariance for \eqref{sys} holds if and only if
\[
D_1h^{[2]}(x,\pi)h^{[1]}(x,\pi)+D_2h^{[2]}(x,\pi)h^{[2]}(x,\pi)=0\text{  on  }U_\pi.
\]
This is, by construction, equivalent to $h(x,\pi)=h_{\rm red}(x,\pi)$ on $U_\pi$. 
\end{proof}

\begin{example} Consider the irreversible Michaelis-Menten system \eqref{mimeirrev}, with $\pi=(e_0, k_1, k_{-1}, k_2)$.
\begin{enumerate}[(a)]
\item With QSS for complex, one has a QSS parameter value $\pi^*=(0,k_1,k_{-1}, k_2)$ with all $k_i>0$, since 
\[
h^{[2]}=-(k_1s+k_{-1}+k_2)\cdot c
\]
with $\pi=\pi^*$. The variety  $U_{\pi^*}$ is defined by $c=0$ and clearly invariant (and the rank condition is also satisfied).
\item Now consider QSS for substrate $s$. (We rearrange variables to $(c,s)$ in order to remain within the notational framework introduced in subsection \ref{subs32}.) Here $\pi^*=(e_0, k_1,0,k_2)$ with positive entries $e_0, k_1,k_2$ is a QSS parameter value for $s$, since
\[
\dot s=h^{[2]}=-k_1(e_0-c)\cdot s
\]
and the QSS variety, which is characterized by $s=0$, is invariant.
\end{enumerate}
At this point we do not (yet) address the question how QSS parameter values can be determined; see subsection \ref{compiss} below.
\end{example}
As in these examples, QSS parameter values frequently describe degenerate settings which, by themselves, seem of little interest for applications. (For instance, an enzyme reaction with zero enzyme concentration is hardly relevant.) But small perturbations of such degenerate settings turn out to be relevant.\\
Next we will therefore establish that for parameters near a QSS parameter value one has approximate accuracy, which is hardly surprising. More importantly, on the other hand we will obtain lower bounds for the norm of the difference of solutions of  \eqref{sys} and of  \eqref{qssredsys} with initial value in $U_\pi$ whenever $\pi$ is not a QSS parameter value. 

The proofs of the following statements are rather elementary, and are based on familiar theorems. We move them to the Appendix, \ref{depresults}, because the technicalities are not relevant for the focus of the present paper. However, we will state the relevant conditions and facts in detail.\\
Thus consider equation \eqref{sys} and the reduced system \eqref{qssredsys} on a suitable compact set $K^*\subseteq \mathbb R_+^n\times\mathbb R_+^m$ with nonempty interior. For the remainder of this subsection, norm always means the maximum norm, resp. the corresponding operator norm. By $\overline {B_r(y)}$ we denote the closed ball in $\mathbb R^n$ with center $y$ and radius $r>0$. 
\begin{itemize}
\item  We assume that $h^{[2]}(\widehat y,\widehat\pi)=0$ for some $(\widehat y,\widehat\pi)$ in the interior ${\rm int}\, K^*$.
\item We assume that $D_2h^{[2]}(x,\pi)$ is invertible  for all $(x,\pi)\in K^*$.
\item We assume that there exist $y_0\in\mathbb R^n$ and $r>0$ with the following property: Whenever $(x,\pi)\in K^*$ for some $x\in \mathbb R^n$ and some $\pi\in\mathbb R^m$ then  $\overline {B_r(y_0)}\times\{\pi\}\subseteq K^*$.
\item Let $R>0$ such that $\Vert h(x,\pi)\Vert\leq R$ and $\Vert h_{\rm red}(x,\pi)\Vert\leq R$  for all $(x,\pi)\in K^*$.
\item Let $L>0$ such that $\Vert Dh(x,\pi)\Vert\leq L$ and  $\Vert Dh_{\rm red}(x,\pi)\Vert\leq L$  for all $(x,\pi)\in K^*$.
\end{itemize}
These conditions imply that every $U_\pi$, with $\pi$ near $\widehat\pi$, is a submanifold. Note that every $(y_0,\pi_0)$ with $y_0$ in the interior of $\mathbb R^n_+$ is contained in some $K^*$ that satisfies the last three of the above conditions. \\
\begin{proposition}\label{consapprox}Assume that the above conditions are satisfied for $K^*$.
\begin{enumerate}[(a)]
\item Let $\pi$ be given such that $U_\pi\times\{\pi\}$ has nonempty intersection with ${\rm int}\,K^*$, let $(y,\pi)$ be a point in this intersection and $V_\pi\subseteq \mathbb R^n$ be some open neighborhood of $y$ such that $(V_\pi\cap U_\pi)\times\{\pi\}\subseteq K^*$. Moreover let $T>0$ such that the solution of \eqref{sys} with initial value $y$ exists and remains in $V_\pi$ for all $t\in[0,\,T]$. Then there exists a compact neighborhood $A_\pi\subseteq V_\pi$ of $y$ with the following properties:  (i) For every $z\in A_\pi$ the solution of \eqref{sys} with initial value $z$ exists and remains in $V_\pi$ for all $t\in[0,\,T]$. 
(ii) For every $\epsilon>0$ there is a $\delta_1>0$ such that the solution of \eqref{qssredsys} with initial value $z\in A_\pi\cap U_\pi$ exists and remains in $V_\pi$ for $0\leq t\leq T$ whenever $\Vert h-h_{\rm red}\Vert<\delta_1$ on $V_\pi$.
(iii) For every $\epsilon>0$ there is a $\delta\in (0,\,\delta_1]$  such that the difference of the solutions of \eqref{sys} resp. of  \eqref{qssredsys} with initial value $z\in A_\pi\cap U_\pi$ has norm less than $\epsilon$ for all $t\in \left[0,T\right]$ whenever  $\Vert h-h_{\rm red}\Vert<\delta$ on $V_\pi$.
\item Let $y \in U_\pi$  and let $\rho_0>0$ such that 
\[
\overline{B_{\rho_0/2L}(y)}\times\{\pi\}\subseteq K^*.
\]
Let $\rho\leq \rho_0$ such that $\Vert h(y,\pi)-h_{\rm red}(y,\pi)\Vert\geq2\rho$. Then for $t^*:=\rho/(2LR)$ the solutions of \eqref{sys} resp. of \eqref{qssredsys} with initial value $y$  exist and remain in $B_{\rho_0/2L}(y)$ for $0\leq t\leq t^*$, and their difference has norm at least $\rho^2/(2LR)$ at $t=t^*$.
\end{enumerate}
\end{proposition}
\begin{proof} Part (a)  is a direct consequence of e.g. Walter \cite{Walter}, \S12 VI; for part (b) see the Appendix, \ref{depresults}.
\end{proof}
\begin{corollary} Let $(y^*,\pi^*)\in U_{\pi^*}\times\{\pi^*\}$ be given such that $y^*$ lies in the open positive orthant. Let $V\subseteq\mathbb R^n$ be a neighborhood of $y^*$ with $\overline V\times\{\pi^*\}\subseteq K^*$, let $A\subseteq V$ be a compact neighborhood of $y^*$, let $B\subseteq \mathbb R^m$ be a compact neighborhood of $\pi^*$. Moreover let $T>0$ such that any solution of \eqref{sys} with initial value $z\in A$ and parameter $\pi\in B$ exists and is contained in $V$ for $0\leq t\leq T$.
Then the following are equivalent.
\begin{enumerate}[(a)]
\item For any positive integer $k$ and any $\delta >0$ there exists $\pi_k\in B$ with $\Vert \pi_k-\pi^*\Vert <\delta$ such that the solution of \eqref{qssredsys} with initial value $z\in U_{ \pi_k}\cap V$ exists and remains in $V$ for $0\leq t\leq T$, and its difference to the solution of \eqref{sys} with the same initial value has norm less than $1/k$ for all $t\in \left[0,T\right]$.
\item $\pi^*$ is a QSS parameter value.
\end{enumerate}
\end{corollary}
\begin{example}We continue the example following Proposition \ref{consexact}, with the irreversible Michaelis-Menten system \eqref{mimeirrev}, and $\pi=(e_0, k_1, k_{-1}, k_2)$.
\begin{enumerate}[(a)]
\item A small perturbation of the QSS parameter value $\pi^*=(0,k_1,k_{-1}, k_2)$ for complex, with all $k_i>0$, yields a parameter value with small $e_0$, and $U_\pi$ is defined by
\[
k_1e_0s-(k_1s+k_{-1}+k_2)\cdot c=0;
\]
which is the familiar version of the QSS variety; at the QSS parameter value $\pi^*$ this degenerates into $c=0$.
Solving this for $c$ and substituting, one recovers the familiar one dimensional Michaelis-Menten equation, and approximate accuracy holds due to continuous dependence. \\
(We restricted attention to a small perturbation of a particular kind here, changing only the first entry of the parameter value from zero to a positive value and leaving the other entries -- which are assumed positive a priori -- unchanged. A more general perturbation would change the first entry from zero to some multiple of a small parameter $\varepsilon$, and also change the other entries by an order $\varepsilon$ term. The net result for the reduced equation (up to higher order terms in $\varepsilon$) would be unchanged. See also the corresponding discussion in \cite{gwz}, e.g. subsection 3.5. In subsequent examples we will take similar shortcuts.) 
\item Considering QSS for substrate $s$, we look at a small perturbation of $\pi^*=(e_0, k_1,0,k_2)$ with positive entries $e_0, k_1,k_2$, hence small $k_{-1}$. The QSS variety $U_\pi$ is defined by
\[
-k_1e_0s +(k_1s +k_{-1})c=0,
\]
and the reduced equation (after rewriting) is given by
\[
\dot c=-k_2c;
\]
again with approximate accuracy due to continuous dependence.
\end{enumerate}
\end{example}

Notions related to approximate invariance are not new in QSS discussions. Schauer and Heinrich \cite{hs-inv} proposed an argument of this type for the irreversible Michaelis-Menten system with QSS for complex. They argued that, to ensure approximate validity of the QSS reduction, the solution trajectory should remain close to the QSS variety defined by ``$\dot c=0$", and they obtained conditions on the parameters from this observation.
Their line of reasoning was later taken up (using somewhat different ``infinitesimal" conditions) and expanded in \cite{nw09}, as well as in \cite{cgw}, Section 4. \\
Essentially we argue in a similar manner in the present paper, but we reverse the argument. Instead of requiring a priori the (approximate) invariance of the manifold $U_\pi$, as Schauer and Heinrich did, we focus on the (approximate) accuracy of the classical QSS-reduction procedure which, after all, is the primary objective. Eventually, as we have seen, both requirements lead to the same conditions. (In contrast, in their definition of validity for QSS, Kollar and Siskova \cite{kosi} require a less restrictive invariance condition but a more restrictive convergence condition. Expressed in the terminology used in the present paper, they do not require invariance of $U_{\pi^*}$ but stability and exponential attractivity for all initial values on $U_{\pi^*}$.)
\begin{remark}\label{qssdefrem}
\begin{enumerate}[(a)]
\item For a QSS parameter value $\pi^*$ system \eqref{sys} admits, by definition, the invariant manifold $U_{\pi^*}$. But the existence of a nearby invariant manifold for systems \eqref{sys} with $\pi$ near $\pi^*$ is not guaranteed unless certain additional conditions hold (see e.g. Fenichel \cite{Feninv} and the CSPT approach by Lam and Goussis \cite{lamgo}). Below (see Subsection \ref{subs35} and Section \ref{secstructure}) we will consider cases where the existence of invariant manifolds is assured.
\item We did not (yet) refer to singular perturbations. These are highly relevant, but our focus in this section is on the minimal requirement for the classical QSS reduction procedure. In turn, this focus on a minimal requirement implies that some QSS parameter values may provide a poor approximation from a practical point of view. (One example was mentioned in Subsection \ref{subsecsubs}.) In Section \ref{secstructure} we will see how singular perturbation scenarios are frequently a natural consequence of QSS assumptions for reaction equations.
\item Moreover, we did not require attractivity of the QSS variety (or some other manifold), or invoke time scale arguments, which form the basis of Segel and Slemrod's work \cite{SSl}.
\end{enumerate}
\end{remark}
To summarize this subsection: It seems justified to investigate QSS reduction only in the neighborhood of QSS parameter values, and we will do so in the following. But by themselves QSS parameter values are just a necessary ingredient for application-relevant reduction, not a sufficient one. (At this point time-scale arguments may be useful when investigating relevance for applications.) On the plus side, QSS parameter values are amenable to algorithmic algebra (as will be seen next), and a case-by case analysis of the associated systems is possible.
\subsection{Computational issues}\label{compiss}

Given a parameter dependent reaction system, it is a typical and important question to ask for parameter values at which QSS takes place (see Schauer and Heinrich \cite{hs-inv}, Segel and Slemrod \cite{SSl}).
Therefore it is a welcome property of QSS parameter values that they can be characterized by algebraic means (polynomial equations and inequations) and computed with the help of algorithmic algebra, as was noticed in \cite{gwz2}. We present here the underlying reason why this works and sketch the path toward an algorithmic determination of QSS parameter values.\\
The crucial point is invariance of the QSS variety $U_{\pi^*}$ which is defined by $h_{r+1}=\cdots=h_n=0$. The invariance condition (locally) can be expressed via the existence of relations
\[
L_h(h_k)=\sum_\ell \mu_{k\ell}h_\ell\, ,\quad r+1\leq k\leq n
\]
with rational $\mu_{k\ell}$ that are defined at the point in question. (This just another way to express tangency of the vector field to the variety.)
The following Proposition builds on this observation; it is  a modification and extension of \cite{gwz2}, Proposition 5.
\begin{proposition}\label{adhoccrit} Let the polynomial  system \eqref{sys} be given, with notation and conditions as in Definitions \ref{qssbase} and \ref{QSSdef}; in particular let $\pi^*$ be a QSS parameter value and $(y^*,\,\pi^*)\in U_{\pi^*}$.
\begin{enumerate}[(a)]
\item Then $(y^*,\,\pi^*)$ is a common zero of $h_{r+1},\ldots,h_n$, their Lie derivatives $L_h(h_{r+1}),\ldots,L_h(h_n)$ and all $(n-r+1)\times(n-r+1)$ minors of the matrices
\[
 A_k:=\begin{pmatrix} Dh_{r+1}\\
                       \vdots\\
                      Dh_{n}\\
                       DL_h(h_k)
\end{pmatrix}, \quad r+1\leq k\leq n.
\]
(As before, $D$ denotes the derivative with respect to $x$.)
\item Conversely, if $(\widehat y,\,\widehat\pi)$ is a common zero of the polynomials above, and if the rank of $D_2h^{[2]}(\widehat y,\,\widehat\pi)$ is equal to $n-r$ then $\widehat \pi$ is a QSS parameter value.
\end{enumerate}
\end{proposition}
\begin{proof} The proof of part (a) is essentially as in \cite{gwz2}, Proposition 5. We give a sketch for the reader's convenience.
By the invariance criteria in Lemma \ref{invlem}, invariance of $U_{\pi^*}$ implies the existence of rational functions $\mu_{k\ell}$ which are regular at $(y^*,\pi^*)$ such that
\[
L_h(h_k)=\sum_\ell \mu_{k\ell}h_\ell\, ,\quad r+1\leq k\leq n, 
\]
and therefore $(y^*,\pi^*)$ is a common zero of the $h_k$ and the $L_h(h_k)$. Moreover this relation implies
\[
DL_h(h_k)=\sum_\ell(D \mu_{k\ell}h_\ell+ \mu_{k\ell}Dh_\ell)
\]
and
\[
DL_h(h_k)(y^*,\pi^*)=\sum_\ell \mu_{k\ell}(y^*,\pi^*)Dh_\ell(y^*,\pi^*),
\]
which shows that the matrix $A_k$ has rank $\leq n-r$.\\
To prove part (b) it suffices to show the existence of analytic functions $\nu_{k\ell}$ near $(\widehat y,\widehat \pi)$ such that 
\[
L_h(h_k)=\sum_\ell \nu_{k\ell}h_\ell\, ,\quad r+1\leq k\leq n.
\]
This is an immediate consequence of the following\\
{\em Fact. Let $\widetilde W\subseteq \mathbb K^n$ open, $z\in\widetilde W$ and moreover $s<n$ and $\theta_1,\ldots,\theta_s$ analytic on $\widetilde W$, with Jacobian of rank $s$ throughout, and $\theta_1(z)=\cdots=\theta_s(z)=0$. Denote by $\widetilde Z$ the common zero set of the $\theta_i$. If $\psi$ is analytic on $\widetilde W$ and 
\[
{\rm rank}\begin{pmatrix}D\theta_1(x)\\ \vdots\\ D\theta_s(x)\\D\psi(x)\end{pmatrix} =s \text{  for all  }x\in \widetilde Z
\]
then there exists a neighborhood of $z$, $\alpha\in\mathbb K$ and analytic functions $\mu_i$ such that
\[
\psi=\alpha+\sum_{i=1}^s\mu_i\theta_i
\]
(In particular, $\psi$ is constant on $\widetilde Z$.)}\\
To prove this claim we may assume that $z=0$ and $\theta_i=x_i$ for $1\leq i\leq s$. Then the condition on the Jacobian is equivalent to
\[
\frac{\partial\psi}{\partial x_1}(x)=\cdots=\frac{\partial\psi}{\partial x_s}(x)=0\text{  for all  }x\in \widetilde Z.
\]
Given the Taylor expansion 
\[
\psi=\sum\alpha_{i_1,\ldots,i_n}x_1^{i_1}\cdots x_n^{i_n}
\]
this implies that $\alpha_{i_1,\ldots,i_n}=0$ whenever $i_1+\cdots +i_n>0$ and $i_1+\cdots +i_s=0$. Hence nonconstant monomials  with nonzero coefficients are multiples of some $x_i$ with $1\leq i\leq s$. The claim follows. (For the smooth case one obtains a proof by invoking a theorem of Hadamard.)

\end{proof}
\begin{definition} If $(\widehat y,\,\widehat\pi)$ is a common zero of the polynomials in Proposition \ref{adhoccrit} (not necessarily satisfying any rank condition) then we call $\widehat \pi$ a {\em QSS-critical parameter value}. 
\end{definition}

The applicability of Proposition \ref{adhoccrit} for the computation of QSS(-critical) parameter values is intuitively clear: Fix $j$ with  $r+1\leq j\leq n$. Then the points $(y,\pi^*)$ of $U_{\pi^*}$ satisfy the  $n-r+1$ equations $h_{r+1}=\cdots =h_{n}=L_h(h_j)=0$, and moreover the determinant conditions involving $L_h(h_j)$, of which there are at least $r$. One therefore has an overdetermined system of at least $n+1$ equations for the $n$ entries of $y$, which one expects to admit a solution only for certain parameter values. In turn, this fact can frequently be used to determine QSS-critical parameter values, and standard algorithms in computational algebra (employing elimination ideals) are applicable (see \cite{gwz2}, in particular Section 7, for more details). 
\begin{example} Write the irreversible Michaelis-Menten system \eqref{mimeirrev} as $\dot x=h(x,\pi)$. 
To find QSS-critical parameter values for substrate $s$, consider 
\[
\theta:=h_1=L_h(s),\quad L_h(\theta)=-(k_1(e_0-c)+k_1s+k_{-1})\theta-(k_1s+k_{-1})k_2c
\]
and their Jacobian determinant. A computation (using the ideal generated by these three polynomials and standard software) similar to \cite{gwz2}, Example 4 shows that any QSS-critical parameter value
$\pi^*=(e_0^*,k_1^*,k_2^*,k_{-1}^*)$ must have (at least) one entry $0$. Here an advantage of the ``classical'' approach becomes apparent: Focussing on QSS parameter values yields a complete list of candidates for application-relevant QSS reduction.
\end{example}
From an algebraic perspective it is natural to consider not only the polynomials listed in part (a) of the Proposition but rather the ideal $J\subseteq \mathbb R[x,\,\pi]$ generated by these polynomials; see more about this in the Appendix, \ref{algoapp}. 

\subsection{An intermediate resum\'e}\label{subs35}
We now take a closer look at the QSS reduction near a QSS parameter value, and once more investigate the accuracy of the approximation. Proposition \ref{consapprox} relies on standard continuous dependency results, but this may be too weak for some systems.\\
 To illustrate the possible problem, fix a parameter value $\pi$ and a QSS parameter value $\pi^*$, write $\rho:=\pi-\pi^*$ and consider Taylor expansions of $h(x,\pi^*+\delta\rho)$ and $h_{\rm red}(x,\pi^*+\delta\rho)$ up to first order in $\delta$. With the abbreviations
\[
h(x,\pi^*)=h_0(x),\quad h(x,\pi)=h_0(x)+\delta h_1(x)+\cdots,
\]
and similar expansions for $h^{[1]}$ and $h^{[2]}$, the QSS reduction up to first order in $\delta$ is given by
\begin{equation}\label{qssredsysdel}
\begin{array}{rcl}
\dot x^{[1]}&=& h_0^{[1]}(x)+\delta h_1^{[1]}(x)+\cdots\\
\dot x^{[2]}&=& -D_2h_0^{[2]}(x)^{-1}D_1h_0^{[2]}(x)h_0^{[1]}(x) +\delta q(x)+\cdots
\end{array}
\end{equation}
with 
\[
\begin{array}{cr}
 q(x)=& \left(D_2h_0^{[2]}(x)\right)^{-1}D_2h_1^{[2]}(x)\left(D_2h_0^{[2]}(x)\right)^{-1}D_1h_0^{[2]}(x)h_0^{[1]}(x)\\
  & -\left(D_2h_0^{[2]}(x)\right)^{-1}D_1h_1^{[2]}(x)h_0^{[1]}(x)\\
 &  -\left(D_2h_0^{[2]}(x)\right)^{-1}D_1h_0^{[2]}(x)h_1^{[1]}(x)
\end{array}
\]
This reduction is robust with respect to changes in the ``small parameter" $\delta$ if $ h_0$ has only isolated zeros on $U_{\pi^*}$ and the stationary points of \eqref{sys} on $U_{\pi^*}$ are hyperbolic. (For instance, near a nonstationary point on $U_{\pi^*}$, a local parameterization and a flow-box argument show that there is a local invariant manifold of dimension $r$ for $\delta$ near $0$ and that this invariant manifold is close to $U_{\pi^*}$.)\\
Matters may be different in the 
{\em singular setting} (following the terminology in Fenichel \cite{fenichel}), when $ h_0$ has non-isolated zeros on $U_{\pi^*}$. 
For the purpose of illustration we just consider the {\em fully singular setting} here: When $h_0$ vanishes on $U_{\pi^*}$ then we have the QSS reduction
\begin{equation}\label{qssredsysing}
\begin{array}{rcl}
\dot x^{[1]}&=& \delta h_1^{[1]}(x)+\cdots\\
\dot x^{[2]}&=& -\delta D_2h_0^{[2]}(x)^{-1}D_1h_0^{[2]}(x)h_1^{[1]}(x) +\cdots,
\end{array}
\end{equation}
and the expansion of $h$ for any point on $U_{\pi^*}$ also starts with terms of order $\delta$. Since Proposition \ref{consapprox}(a) guarantees a correct approximation only up to errors of order $\delta$,  the QSS reduction may become unreliable here. (For a clear description of the underlying problem see Stiefenhofer \cite{sti}, p.~595ff.) 
\begin{example}
Consider the irreversible Michaelis-Menten equation \eqref{mimeirrev} with slow product formation (i.e., small parameter $k_2$). Thus $\pi^*:=(e_0,k_1, k_{-1}, 0)^{\rm tr}$ with positive $e_0,k_1,k_{-1}$ is a QSS parameter value for complex concentration $c$, and we set $\rho:=(0,0,0,1)^{\rm tr}$ and $\delta:=k_2$ in accordance with the notation above. This scenario also admits a singular perturbation (Tikhonov-Fenichel) reduction with small parameter $k_2$, and it is known (see \cite{GSWZ1}, 3.1  and \cite{gw}, Example 8.6) that the reduced equation (after rewriting as a one-dimensional system) is given by
\[
\dot s=-\frac{k_2k_1e_0s}{k_1k_{-1}e_0/(k_1s+k_{-1})+(k_1s+k_{-1})},
\]
with convergence guaranteed by Tikhonov's theorem. On the other hand, classical QSS reduction for complex yields
\[
\dot{s}=-\frac{k_2k_1e_0s}{k_1s+k_{-1}+k_2}=-\frac{k_2k_1e_0s}{k_1s+k_{-1}}+\cdots
\]
(up to higher order in $k_2$). In the slow time scale $\tau=k_2t$ one has
\[
s^\prime=-\frac{k_1e_0s}{k_1k_{-1}e_0/(k_1s+k_{-1})+(k_1s+k_{-1})}\text{  versus   }s^\prime=-\frac{k_1e_0s}{k_1s+k_{-1}}.
\]
Since we excluded the case that  $k_{-1}e_0$ is also of order $\delta$, the  QSS reduction procedure yields an incorrect result, predicting too slow decay of substrate. One can verify this in numerical experiments, but one has to be mindful that the reduction should be expected to be valid only on the QSS variety $U_{\pi^*}$, which is defined by $k_1e_0s+(k_1s+k_{-1})c=0$ and coincides with the asymptotic slow manifold. Therefore one has to choose starting values accordingly. (If one wishes to investigate system \eqref{mimeirrev} with the usual initial value $(s_0,0)$ then one has to consider the fast time scale first and determine an appropriate starting value on the slow manifold; see \cite{gw}, subsection 2.3. With the incorrect starting value $s_0$ both reductions will provide bad approximations.)\\
One may note here that the case of small $k_2$ is in fact involving slow and fast reactions, thus properly belongs into the realm of partial equilibrium approximation (PEA). But considering the mathematical side, such parameters are close to a QSS parameter value, and therefore they should be discussed in the QSS context if only to show that QSS reduction is inappropriate,  and singular perturbation reduction is appropriate.
\end{example}

\section{Structure and singular perturbations}\label{secstructure}
In many applications, classical QSS assumptions lead to singular perturbation scenarios, although there seems to be no a priori reason for this. In the present section we will provide some evidence that QSS assumptions {\em for reaction networks} naturally  lead to singular perturbation settings. The underlying reason is that invariance of the QSS variety implies the existence of further invariant varieties, with a possible exception when this variety is an affine subspace defined by $x_i={\rm const}$ for all QSS species $x_i$. (We call such varieties affine coordinate subspaces.) Since many reaction equations have the property that every forward invariant set in the positive orthant contains a stationary point, one automatically arrives at a singular scenario with non-isolated stationary points whenever the QSS variety is not an affine coordinate subspace for the QSS species.\\
 Assuming the conditions which guarantee Tikhonov-Fenichel reductions, we proceed to compare these to QSS reductions. It turns out that they do not match in general (which makes such QSS reductions questionable), but they do match up to first order when the QSS variety is an affine coordinate subspace. This provides an explanation why QSS reduction works well for Michaelis-Menten with small enzyme concentration.\\
As a further application we show that a QSS reduction (up to first order in a small parameter) can be computed explicitly, and agrees with singular perturbation reduction, even in cases when the algebraic obstacles to explicitly solving $h^{[2]}=0$ are insurmountable.
\subsection{The structure of the QSS variety}
Throughout this subsection let the situation of Definition \ref{qssbase} be given and assume that the hypotheses of Proposition \ref{adhoccrit} hold. Thus $\pi^*$ is a QSS parameter value, and the QSS variety $U_{\pi^*}$ is invariant. Due to the particular build of the QSS variety, we will find that there exist further invariant varieties. We introduce some convenient notation first.
\begin{definition}\label{affinedef}
\begin{enumerate}[(i)]
\item Given $\gamma:=(\gamma_{r+1},\ldots,\gamma_n)\in\mathbb R^{n-r}$, let
\begin{equation}\label{zdef}
\begin{array}{rcl}
\psi_{j, \gamma}(x)=\psi_j(x)&:=&x_j-\gamma_j,\quad r+1\leq j\leq n, \text{  and}\\ 
Z_\gamma&:=&\left\{y;\, \psi_{r+1}(y)=\cdots=\psi_n(y)=0\right\}.
\end{array}
\end{equation}
We call $Z_\gamma$ an {\em affine coordinate subspace.}
\item We say that system \eqref{sys} {\em admits a QSS reduction to an affine coordinate subspace} if $U_{\pi^*}\subseteq Z_{\gamma^*}$ for some $\gamma^*\in \mathbb R^{n-r}$.
\end{enumerate}
\end{definition}
We note that $Z_{\gamma^*}$ is invariant for \eqref{sys} whenever (ii) of Definition \ref{affinedef} holds, because $U_{\pi^*}$ is open in   $Z_{\gamma^*}$ and the Zariski closure of an invariant set of \eqref{sys} is invariant.\\
The intersections of a QSS variety with corresponding affine coordinate subspaces are again invariant; this fact explains their relevance:
\begin{proposition}\label{zsetsprop}
For every $j$, $r+1\leq j\leq n$ the set 
\[
U_{\pi^*}\cap\left\{y;\,\psi_{j, \gamma}(y)=0\right\}
\]
is invariant for $\dot x=h(x,\pi^*)$. In particular, 
for every $\gamma=(\gamma_{r+1},\ldots,\gamma_n)$ the set $U_{\pi^*}\cap Z_\gamma$ is invariant for system \eqref{sys} with $\pi=\pi^*$.
\end{proposition}
\begin{proof} If the intersection is empty then there is nothing to prove. Otherwise, due to the invariance of $U_{\pi^*}$ there exist rational functions $\mu_{jk}=\mu_{jk}(x)$ which are regular on an open-dense subset of $U_{\pi^*}$ such that 
\[
L_h(h_i)=\sum_k\mu_{ik}h_k,\quad r+1\leq i\leq n;
\]
see Appendix \ref{invapp}, Lemma \ref{invlem}. By definition one has
\[
L_h(\psi_{j, \gamma})= h_j,\quad r+1\leq j\leq n.
\]
Taking these conditions together, Lemma \ref{invlem} shows the assertions.
\end{proof}
\begin{example}
Consider again the irreversible Michaelis-Menten equation \eqref{mimeirrev} with slow product formation (small parameter $k_2$) and QSS parameter value $\pi^*:=(e_0,k_1, k_{-1}, 0)^{\rm tr}$ for complex concentration. The QSS variety $U_{\pi^*}$ is then given by the equation
\[
k_1e_0s-(k_1s+k_{-1})c=0.
\]
By Proposition \ref{zsetsprop} each intersection with a level set $c=\gamma\geq 0$ is a point (necessarily stationary), and we have a singular scenario.
\end{example}
We note some consequences.
\begin{corollary}\label{distinction} Let $\pi^*$ be a QSS parameter value of system \eqref{sys}.
\begin{enumerate}[(a)]
\item Let $j$ be fixed. If $U_{\pi^*}\cap\left\{y;\,y_j-\gamma_j=0\right\}\not=\emptyset$ for more than one $\gamma_j$ then $x_{j}$ is a first integral for the restriction of \eqref{sys} to $U_{\pi^*}$; i.e. the intersections with all level sets have smaller dimension than $r$ and they are invariant for \eqref{sys}. Otherwise $U_{\pi^*}$ is contained in some hyperplane $\{x;\,x_j=\gamma_j^*\}$ for a unique $\gamma_j^*$.
\item  If the rank of the Jacobian of $(\psi_{r+1},\ldots,\psi_n,h_{r+1},\ldots,h_n)$ equals  $n$ at one point of $U_{\pi^*}$ then every point of $U_{\pi^*}$ is stationary.
\item If  $U_{\pi^*}$ is a curve then $U_{\pi^*}$  is open-dense in a coordinate subspace (thus all but one of the $x_i$ is constant), or every point of this curve is stationary.
\end{enumerate}
\end{corollary}
\begin{proof} (a) By irreducibility, unless $U_{\pi^*}$ is contained in $\left\{y;\,\psi_j(y)=0\right\}$ for some $\gamma_j^*$, the dimension 
of the intersection is less than $r$.\\
As for part (b), full rank of the Jacobian in one point of $U_{\pi^*}\cap Z_{\gamma}$ implies full rank in an open and dense subset. In this subset, $U_{\pi^*}\cap Z_{\gamma}$ locally contains just single points, and by invariance these points must be stationary.  Due to irreducibility, every point of $U_{\pi^*}$ is stationary. Part (c) is proven by a similar argument.
\end{proof}
Part (c) of the Corollary is quite relevant for applications, which frequently consider reduction to dimension one. For reaction systems, we can place this observation in a broader context.
\begin{remark}
Proposition \ref{zsetsprop} seems to provide an explanation for the ubiquity of singular perturbations in QSS {\em for reaction systems}. Indeed, assume that $\pi^*$ is a QSS parameter value but the corresponding QSS variety is not open-dense in an affine coordinate subspace. Then there are infinitely many $\gamma$ such that  $U_{\pi^*}\cap Z_{\gamma}\not=\emptyset$. For many classes of reaction systems all physically relevant forward invariant sets contain a stationary point; therefore one may expect $U_{\pi^*}\cap Z_{\gamma}$ to contain a stationary point for infinitely many $\gamma$. 
\end{remark}
According to Fenichel \cite{fenichel}, one characteristic of singular perturbation settings is the existence of non-isolated stationary points. Therefore we define:
\begin{definition}\label{tfcritdef}
We call a parameter value $\widehat\pi$ {\em TF-critical} (or, at length, Tikho\-nov-Fenichel-critical) whenever $Y_{\widehat\pi}$ contains non-isolated stationary points. (In other words, $Y_{\widehat\pi}$ contains a positive dimensional subvariety of stationary points.)
\end{definition}
For instance, in the situation of part (b) of the above Corollary, $\pi^*$ is a TF-critical parameter value. The notion of TF-critical parameter value is a precursor to the notion of TF (Tikhonov-Fenichel) parameter value introduced in \cite{gwz} (see Appendix, \ref{sptredu} for more details). At TF parameter values the system admits a singular perturbation reduction according to Tikhonov's theorem. \\

Determining QSS parameter values which admit reduction to an affine coordinate subspace is less computationally involved than for general varieties; details are given in the Appendix, \ref{algoapp}.

\subsection {Singular perturbation scenarios}\label{subsqssspt}
As a direct consequence of the definitions, Tikhonov-Fenichel-critical parameter values are also QSS parameter values (with respect to any set of variables). In turn, certain QSS-critical parameter values may be TF-critical by the observations in the previous subsection. \\
If some $\widehat \pi$ is actually a Tikhonov-Fenichel parameter value then, on the one hand, validity of the singular perturbation reduction is ensured. But on the other hand, this reduction (see Appendix, \ref{sptredu}, equation  \eqref{redform}) need not agree with the QSS reduction \eqref{qssredsys}, as shown by the example at the end of Section \ref{Formalsec}. In the present subsection we will show that under certain conditions (involving coordinate subspaces, most importantly) the two reduction methods yield essentially the same result.\\

There is another reason to give special attention to affine coordinate subspaces, from the perspective of applications and modelling. Indeed the notion of QSS should perhaps be reconsidered in a singular perturbation scenario. Given such a setting, one could argue that QSS then holds for {\em all species} in the reduced system on the asymptotic slow manifold (with respect to the original time scale). Indeed, all species change slowly following a short initial phase. But this argument seems to miss the point, since there is no longer a distinguished set of species in quasi-steady state. (A notable exception to this rule occurs, however,  when the QSS variety is an affine coordinate subspace.)\\

For illustration we look again at the example from the end of Section \ref{Formalsec}.
\begin{example}
For the irreversible Michaelis-Menten system \eqref{mimeirrev}, start with a QSS assumption for complex $c$. Then $(e_0,k_1,0,k_{-1}$) is a QSS parameter value, therefore we may consider the ``small parameter" $k_2$. This QSS parameter value is also a TF parameter value, and singular perturbation reduction yields
\[
\frac{d}{dt}\begin{pmatrix}s\\c\end{pmatrix}=\frac{-k_2c}{k_1(e_0-c)+k_1s+k_{-1}}\begin{pmatrix}k_1s+k_{-1}\\k_1(e_0-c)\end{pmatrix}
\]
(see \cite{GSWZ1}, 3.1) on the invariant curve determined by $k_1e_0s-(k_1s+k_{-1})c=0$. \\
For the system on the curve the rates of change for $s$ and $c$ are of the same order. Therefore the QSS assumption for $c$ cannot be validated for the reduced system, which correctly describes the dynamics after a short initial phase.
(As we have seen earlier, the classical QSS reduction is different from the singular perturbation reduction here, hence yields incorrect results.)
\end{example}
Thus, while QSS reduction will frequently lead to singular scenarios, the QSS variety and the slow manifold need not coincide (even locally), and if they do coincide then the reductions may be substantially different, hence classical QSS reduction provides incorrect results. It is therefore appropriate to characterize the distinguished situation when both reductions exist and agree.
\begin{definition}\label{consistency} Assume that system \eqref{sys} admits a QSS parameter value $\pi^*$ for species $x_{r+1},\ldots,x_n$ which is also a TF parameter value. We call the QSS reduction {\em consistent with the singular perturbation reduction} whenever the following hold.
\begin{enumerate}[(i)]
\item The slow manifold $\widetilde V$ and the QSS variety $U_{\pi^*}$ coincide near $y^*$. 
\item Given $\rho$ such that $\pi^*+\epsilon\rho\in \mathbb R^m_+$ for all sufficiently small $\epsilon\geq 0$, the QSS reduction and the Tikhonov-Fenichel reduction of $\dot x=h(x,\pi^*+\epsilon\rho)$ agree up to first order in $\epsilon$.
\end{enumerate}
\end{definition}
Condition (i) is not an automatic consequence of $\pi^*$ being both a QSS parameter value and TF-critical; for an example see the Appendix, \ref{slowqss}.\\

As shown by example at the end of Section \ref{Formalsec}, condition (i) alone does not imply (ii), hence is generally not sufficient for consistency. But we will now prove that (i) implies (ii) in the coordinate subspace scenario. \\
Thus assume that $U_{\pi^*}$ is open and dense in a coordinate subspace, and locally coincides with the slow manifold.
 In order to reduce the notational expenditure, we make some normalizations and ``hide" some parameters. We split $x=(x^{[1]},x^{[2]})$ as usual, and moreover we suppress $\pi^*$ and $\rho$ in the following, showing only $\epsilon$ explicitly. The QSS variety is, by assumption, determined by $x^{[2]}=\gamma^*$ for some $\gamma^*$; for the proof we may assume $\gamma^*=0$.
With these normalizations, and noting that $x^{[2]}=0$ defines an invariant set when $\epsilon=0$, there remains to investigate a system of the form
\begin{equation}\label{linvar}
\begin{array}{rcl}
\dot x^{[1]}&= &B(x^{[1]})x^{[2]}+B^*(x^{[1]},x^{[2]}) +\epsilon\left(u(x^{[1]})+U^*(x^{[1]},x^{[2]}\right) + \cdots \\
\dot x^{[2]}&= &A(x^{[1]})x^{[2]}+A^*(x^{[1]},x^{[2]}) +\epsilon\left(v(x^{[1]})+V^*(x^{[1]},x^{[2]}\right) + \cdots
\end{array}
\end{equation}
in a neighborhood of some point of $U_{\pi^*}$, with terms in the Taylor expansion as follows (all functions being analytic in $x$):
\begin{itemize}
\item For every $x^{[1]}$, the function $u(x^{[1]})$ has values in $\mathbb R^r$, the function $v(x^{[1]})$ has values in $\mathbb R^{n-r}$, and $A(x^{[1]})$ resp. $B(x^{[1]})$ are matrices of appropriate size.
\item $A(x^{[1]})$ is invertible for all $x^{[1]}$.
\item The functions $A^*$  and $B^*$ have order $\geq 2$ in $x^{[2]}$.
\item The functions $U^*$  and $V^*$ have order $\geq 1$ in $x^{[2]}$.
\end{itemize}
The Tikhonov-Fenichel reduction of system \eqref{linvar} is a special case of  \eqref{redform} in the Appendix, \ref{sptredu}, which was determined in \cite{GSWZ1}, Theorem 2; on the slow manifold it is given by
\begin{equation}\label{tflin}
\dot x^{[1]}=\epsilon\cdot\left(u(x^{[1]})-B(x^{[1]})A(x^{[1]})^{-1}v(x^{[1]})\right)
\end{equation}
Generally this does not coincide with the QSS reduction \eqref{qssredsysex}, although there are exceptions (notably the irreversible Michaelis-Menten system for small parameter $e_0$), as remarked in \cite{GSWZ1}.  But these two reductions always are in agreement in their first order terms (which is required in Definition \ref{consistency} and sufficient to ensure convergence), and this is the relevant point.
\begin{proposition}\label{allthesame}
Given system \eqref{linvar}, the first order term in $\epsilon$ of the QSS reduction \eqref{qssredsysex} with respect to $x^{[2]}$ is equal to (the corresponding term in) the Tikhonov-Fenichel reduction \eqref{tflin}. In other words, the QSS reduction is consistent with the singular perturbation reduction whenever the QSS variety is open-dense in a coordinate subspace and concides locally with the slow manifold.
\end{proposition}
\begin{proof} We let 
\[
g(x^{[1]},x^{[2]},\epsilon):=A(x^{[1]})x^{[2]}+A^*(x^{[1]},x^{[2]}) +\epsilon\left(v(x^{[1]})+V^*(x^{[1]},x^{[2]})\right) +\cdots
\]
and note that $g(x^{[1]},0,0)=0$, with invertible $D_2g(x^{[1]},0,0)=A(x^{[1]})$. By the implicit function theorem (with parameter  $x^{[1]}$) we have a solution
\[
x^{[2]}=S(x^{[1]},\epsilon)=S_0(x^{[1]})+\epsilon S_1(x^{[1]})+\cdots
\]
of $g=0$, and one sees $S_0=0$ due to $g(x^{[1]},0,0)=0$. Substitution of this expression into $g=0$ yields
\[
\begin{array}{rcl}
0&=& \epsilon A(x^{[1]})S_1(x^{[1]})+\cdots +A^*(x^{[1]},\epsilon S_1(x^{[1]})+\cdots)\\
 & & +\epsilon v(x^{[1]})+\epsilon V^*(x^{[1]}, \epsilon S_1(x^{[1]})+\cdots)+\cdots
\end{array}
\]
with all the dots representing terms of order $\geq 2$. By construction, $A^*(x^{[1]},\epsilon S_1(x^{[1]})$ and $\epsilon V^*(x^{[1]}, \epsilon S_1(x^{[1]})$ contain only terms of order $\geq 2$. Thus comparing lowest order terms yields $S_1(x^{[1]})=-A(x^{[1]})^{-1}v(x^{[1]})$. In turn, substitution of this expression into the equation
\[
\dot x^{[1]}=B(x^{[1]})x^{[2]}+B^*(x^{[1]},x^{[2]}) +\epsilon\left(u(x^{[1]})+U^*(x^{[1]},x^{[2]})\right)+\cdots
\]
and keeping only the lowest order terms yields, by similar arguments, the assertion.
\end{proof}
Proposition \ref{allthesame} seems to provide a natural explanation why the classical QSS reduction procedure is frequently successful in practice. We are not aware of possible extensions of such a result to more general QSS varieties.

\begin{example} Consider the reversible Michaelis-Menten reaction \eqref{mimerev}, with quasi-steady state for complex and QSS parameter value $e_0$ (all other parameters $>0$); here $x^{[1]}=s$ and $x^{[2]}=c$. With the notation as in \eqref{linvar} we have
\[
\begin{array}{ll}
B(s)=k_1s+k_{-1};& u(s)= -k_1s;\\
A(s) =  -\left(k_1s+k_{-1}+k_2+k_{-2}(s_0-s)\right);& v(s)=k_1s+k_{-2}(s_0-s-c).
\end{array}
\]
With some high-school algebra (but no Taylor expansions) one arrives at the reduced system \eqref{mimerevqssc}.
\end{example}
\subsection{On explicit computation of QSS reductions}
Proposition \ref{allthesame} has a welcome consequence. As has been noted in Pantea et al. \cite{pgrc}, the classical reduction method cannot be put into practice whenever the implicit equations do not admit an explicit solution for $x^{[2]}$ as a function of $x^{[1]}$. (Such settings occur due to Abel's famous theorem on non-solvability of monic polynomial equations by radicals.) But there is a way to circumnavigate this problem if the QSS parameter value is known (as it should be) and the affine coordinate subspace setting is given. Then Proposition \ref{allthesame} allows for a direct computation which requires only basic algebraic operations. 
\begin{example} Consider the following system from Pantea et al. \cite{pgrc}, subsection 2.3:
\[
\begin{array}{rcl}
\dot a&=& k_2by-k_4ax+2k_5z^2\\
\dot b&=& 2k_1y^2-2k_{-1}b^2-k_2by-k_3bz+k_{-3}x^2+k_4ax\\
\dot x &=& 2k_3bz-2k_{-3}x^2-k_4ax\\
\dot y &=&-2k_1y^2+2k_{-1}b^2-k_2by+k_4ax\\
\dot z&=& k_2by-k_3bz+k_{-3}x^2 -2k_5z^2
\end{array}
\]
QSS reduction with respect to $x,\,y,\,z$ leads to a polynomial system which is (generically) not solvable by radicals, as was proven in \cite{pgrc}.\\
But for the QSS parameter value $k_{-1}=0$ (all other parameters $>0$) the system admits the invariant plane given by $x=y=z=0$, and the QSS reduction is consistent with the singular perturbation reduction.
Proposition \ref{allthesame} with small parameter $k_{-1}$ and decomposition
\[
\begin{pmatrix}\dot x\\
                        \dot y\\
                        \dot z
     \end{pmatrix} =\begin{pmatrix}-k_4a& 0 & 2k_3b\\
                                                  k_4a & -k_2b & 0\\
                                                  0 & k_2b & -k_3b\end{pmatrix}\begin{pmatrix} x\\
                         y\\
                         z
     \end{pmatrix} +\begin{pmatrix} \\
                         \vdots\\
                           \\
     \end{pmatrix} +k_{-1}\begin{pmatrix} 0\\
                         2b^2\\
                         0
     \end{pmatrix}
\]
yields the reduced system 
\[
\begin{array}{rcl}
\dot a &=& 2k_{-1}b^2\\
\dot b&=& -2 k_{-1}b^2
\end{array}
\]
It should be noted that Pantea et al. consider the case that both $k_{-1}$ and $k_4$ are small; this would not provide a QSS parameter value since the rank condition  on $D_2h^{[2]}$ from Definition \ref{qssbase} is violated.
\end{example}
Of course, not all QSS reductions of interest lead to affine coordinate subspaces, and therefore Proposition \ref{allthesame} is not a panacea. 
But as we have seen, classical QSS reduction for singular settings may yield incorrect results whenever the QSS variety is not an affine coordinate subspace. Hence there are good reasons to focus on the affine coordinate subspace case, and for this
 we have a feasible alternative approach which avoids any fundamental algebraic obstacles.

\section{Examples and applications}
In this section we discuss various aspects of QSS parameter values, their computation and QSS reduction for several relevant systems.
\subsection{Bimolecular binding with intermediate complex}
Kollar and Siskova \cite{kosi} discuss the reaction network
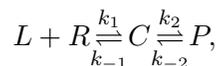
\[
L+R{\overset{k_1}{\underset{k_{-1}}\rightleftharpoons}}
C{\overset{k_2}{\underset{k_{-2}}\rightleftharpoons}} P,
\]
which, via mass action kinetics and stoichiometry, leads to the differential equation system
\[
 \begin{array}{rcl}
 \dot \ell &=& -k_1 \ell (\ell + a) + k_{-1}c\\
 \dot c &=& k_1 \ell (\ell +a) - (k_{-1}+k_2)c + k_{-2} (b-\ell-c)
 \end{array}
\]
with the abbreviations $a:=r(0)-\ell(0)\geq 0$ (with no loss of generality) and $b:=\ell(0)$. We determine all QSS parameter values for this system.\\
In the irreversible case $k_{-2}=0$, which we consider first, one obtains
\[
 \begin{array}{rcl}
 \dot \ell &=& -k_1 \ell (\ell + a) + k_{-1}c\\
 \dot c &=& k_1 \ell (\ell +a) - (k_{-1}+k_2)c.
 \end{array}
\]
For this system the QSS parameter values with respect to $c$ are readily determined via Proposition \ref{adhoccrit}, with the following result.
\begin{center}
\begin{tabular}{|c|c|}
\hline
Condition on parameter & QSS variety ${S} $ defined by\\
\hline
$ k_1 =0$ & $ c=0 $ \\
\hline
$ k_2=0 $ & $k_1\ell (\ell +a) -k_{-1}c =0 $ \\
\hline
$ a =0$ & $ k_1 \ell ^2 -(k_{-1}+k_2)c = 0 $ \\
\hline
\end{tabular}
\end{center}
Here -- and in all following examples -- the understanding is that the remaining parameter values are $>0$.\\ The first two conditions define Tikhonov-Fenichel parameter values, while for the last one ($a=0$) the system admits only an isolated stationary point. For small parameter $a$  the QSS reduction yields (after some simplification) the equation
\[
 \dot \ell= - \frac{k_1 k_2 \ell (\ell +a )}{k_{-1}+k_2}.
\]
For the singular perturbation case of small parameter $k_2$ one obtains the reduced equation
\[
\begin{array}{rcl}
\dot \ell &=&- \frac{k_2 k_{-1} c}{k_{-1}+k_1(a+2\ell)} \\
\dot c &=& -\frac{k_2 k_1(a + 2\ell )c}{a k_1+k_{-1}+2k_1\ell}
\end{array}
\]
on the slow manifold defined by  $k_1 \ell ^2 - k_{-1} c=0$. After simplification one obtains a differential equation for $\ell$ alone, but one should note that this equation -- similar to the situation for Michaelis-Menten -- does not agree (even to first order in $k_2$) with the classical QSS reduction.\\

For the reversible case (i.e., $k_{-2}>0$) one obtains the following list of defining conditions for QSS parameter values with respect to $c$:
\[
k_{-2} = a =0;\quad k_{-2} = k_1 =0;\quad k_{-2} = k_2 =0;\quad  k_{-1} = k_1 =0;\quad k_{-1} = a =0.
\]
Comparison with Kollar and Siskova \cite{kosi} shows that the condition $a=b=0$ (corresponding to small $\ell(0)$ and small $r(0)$ in \cite{kosi}) does not appear. This indicates that the concept of ``validity of QSS reduction" as introduced in \cite{kosi} indeed leads to different parameter regions compared to the QSS parameter approach given here.  (In conjunction with the example in subsection \ref{subsecsubs} one sees that neither definition implies the other.) For the reversible system in question one will generically observe QSS-like behavior locally, near the stationary point $0$. This stationary point is an attracting node, and unless both $|k_{-1}-k_{-2}|$ and $k_2$ are small, the absolute ratio of smaller by larger eigenvalue will be $\ll 1$. Thus the preferred tangent direction for approaching the stationary point will be attained quickly in a suitable neighborhood of $0$. (For the irreversible case one obtains a saddle-node, with the attracting node part containing the first quadrant.) Here we see a relation between QSS and local theory near stationary points (which also seems to reflect the underlying mathematics in some examples from Borghans et al. \cite{borg}). The classical QSS reduction approach in this case (as well as generally) is not suitable for a complete determination of local invariant manifolds. On the other hand, classical QSS reduction works globally when it works.
\subsection{Competitive Inhibition}
The standard model for competitive inhibition (see e.g. Keener and Sneyd \cite{kesn}, p.~13) leads to the differential equation system
\[
 \begin{array}{rcl}
\dot s &=& k_{-1} c_1 - k_1 s (e_0-c_1 - c_2)\\
\dot {c_1} &=& k_1 s (e_0-c_1-c_2) - (k_{-1}+k_2)c_1\\
\dot {c_2} &=& k_3 (e_0 - c_1 - c_2) (i_0 - c_2) -k_{-3}c_2  
 \end{array}
\]
with nonnegative rate constants and initial concentrations $e_0$ for enzyme and $i_0$ for inhibitor. Again we are interested in determining  all QSS parameter values, with various choices for QSS variables. It is known that for small $e_0$ one has Tikhonov-Fenichel reduction with asymptotic slow manifold given by $c_1=c_2=0$; see e.g. \cite{gwz}.\\
If one requires QSS for both complexes $c_1$ and $c_2$ then the determination of QSS parameter values according to Proposition \ref{adhoccrit} yields an elimination ideal with eight generators. (We will not discuss this in detail here, due to space considerations.)\\
If one requires QSSA for the second complex $c_2$ then one finds an elimination ideal with two generators
\[  e_0 i_0 k_1 k_3 k_{-3} (k_{-1}+ k_2), \, e_0 i_0 k_3 k_{-3} (k_3^2(e_0-i_0)^2 + k_{-3}^2 + 2 k_3k_{-3}(e_0+i_0)) (k_{-1}+k_2).
\]
One obtains the following list of QSS parameter values for $c_2$; all varieties have codimension one, one is reducible.  (Positivity may impose additional restrictions, e.g. for the first variety: Whenever $i_0>e_0$ then one ends up with $c_1=c_2=0$.)
\begin{center}
\begin{tabular}{|c|c|}
\hline
Condition on parameter & QSS variety defined by\\
\hline
$ e_0 =0$ & $ k_3 ( c_1 + c_2) (i_0 - c_2) +k_{-3}c_2 =0 $ \\
\hline
$ i_0=0 $ & $ c_2 =0 $ \\
\hline
$ k_3 =0$ & $ c_2 =0 $ \\
\hline
$ k_{-3}=0 $ & $ c_2 = i_0 \text{ or }c_1+c_2=e_0 $  \\
\hline
$ k_{-1}=k_2 $ & $ k_3 (e_0 - c_1 - c_2) (i_0 - c_2) -k_{-3}c_2 =0 $ \\
\hline
\end{tabular}
\end{center}
According to \cite{gwz}, Proposition 8 the first and fourth case correspond to Tikhonov-Fenichel parameter values, the remaining ones do not. We look at one case of QSS reduction: For small $e_0$ (assuming $k_3 (c_1+2 c_2 - e_0-i_0) -k_{-3} \neq 0)$ one obtains the two-dimensional system
\[
 \begin{array}{rcl}
\dot s &=& k_{-1}c_1 - k_1 s(e_0-c_1-c_2)\\
\dot c_1 &=& k_1 s (e_0-c_1-c_2)-(k_{-1}+k_2)c_1\\
\dot c_2 &=& \frac{k_3 (i_0-c_2)(k_1s(e_0-c_1-c_2)-(k_{-1}+k_2)c_1)}{k_3(c_1+2c_2 - e_0-i_0)-k_{-3}}
 \end{array}
\]
on the QSS variety, which may be rewritten as a system for $s$ and $c_1$ after solving a quadratic equation for $c_2$. In this case the one-dimensional asymptotic slow manifold for the singular perturbation reduction is given by $c_1=c_2=0$; the reduced equation was determined in \cite{GSWZ1}, Subsection 3.2. 

\subsection{Cooperativity with an arbitrary number of complexes -- small enzyme concentration}
Here we consider a reversible cooperative reaction network with an arbitrary number $m$ of complexes, with small enzyme concentration and QSS for all complexes. Our goal here is to use the QSS approach in order to compute a singular perturbation reduction (which seems hard to find in a straightforward manner).
With $C_0$ denoting the enzyme we have the network
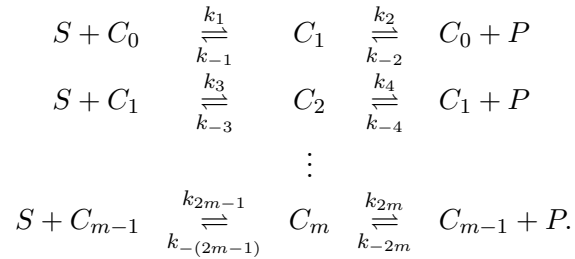
\[
\begin{array}{rcccl}
 S + C_0 &\overset{k_{1}}{\underset{k_{-1}}\rightleftharpoons}& C_1  
&\overset{k_2}{\underset{k_{-2}}\rightleftharpoons}& C_0+P \\
S+C_1 &\overset{k_{3}}{\underset{k_{-3}}\rightleftharpoons}& C_2 
&\overset{k_4}{\underset{k_{-4}}\rightleftharpoons}& C_1+P\\
& & \vdots & &\\
S+C_{m-1} &\overset{k_{2m-1}}{\underset{k_{-(2m-1)}}\rightleftharpoons}& C_m  
&\overset{k_{2m}}{\underset{k_{-2m}}\rightleftharpoons}& C_{m-1}+P.\\
\end{array}
\]
and mass action kinetics yields the differential equation system
\[
\begin{array}{rcl}
 \dot s &=& \sum\limits_{j=0}^{m-1}k_{-(2j+1)} c_{j+1} - k_{2j+1} sc_j \\
 \dot p &=& \sum\limits_{j=0}^{m-1}k_{2j+2}c_{j+1} - k_{-2(j+1)}pc_j\\
 \dot c_0 &=& (k_{-1}+k_2) c_1 - (k_1s-k_{-2}p)c_0\\
 &\vdots&\\
 \dot c_\ell &=& (k_{2\ell-1}s + k_{-2\ell}p)c_{\ell-1} + (k_{-(2\ell+1)}+k_{2\ell+2})c_{\ell+1}\\
& & - (k_{-(2\ell-1)} + k_{2\ell} + k_{2\ell+1} s + k_{-2(\ell+1)}p )c_l,\qquad \qquad  1\leq \ell \leq m-1\\
 &\vdots&\\
 \dot c_m &=& (k_{2m-1}s +k_{-2m}p)c_{m-1}-(k_{-2m-2} + k_{2m})c_m.
\end{array}
\]
The relevant initial values are $s(0)=s_0$, $c_0(0)=e_0$, with all other initial concentrations equal to zero. By stoichiometry one has two first integrals that allow to substitute
\[
\begin{array}{rcl}
c_0&=& e_0-\sum_{j=1}^m c_j,\\
p&=& s_0-s-\sum_{j=1}^m j c_j.\\
\end{array}
\]
As is known from \cite{godiss}, Kap.~5.5 and \cite{GSWZ1}, subsection 3.5, there is a Tikhonov-Fenichel parameter value with $e_0=0$, all other parameters $>0$; the slow manifold is defined by all $c_j=0$ (at least for $s_0$ not too large). Since $e_0=0$ also defines a QSS parameter value, and the QSS variety coincides with the slow manifold, Proposition \ref{allthesame} is applicable. But in this instance we determine the singular perturbation reduction by way of QSS, since inverting the matrix $A(s)$ (notation as in Proposition \ref{allthesame}) would be rather arduous. We first emulate the procedure in \cite{godiss}, Kap.~5.5 (for the irreversible setting) and in a final step we keep only the lowest order terms in the small parameter $e_0$.\\
On the QSS variety one has  ``$\dot {c_0} = 0$", hence
\[
 c_1 = \frac{k_1 s - k_{-2} p}{k_{-1} + k_2} c_0
\]
By induction 
\[
c_{\ell} = c_0 \prod\limits_{j=1}^{\ell} \frac{k_{2j-1} s + k_{-2j} p 
}{k_{-(2j-1)} + k_{2j} }, \qquad 1 \leq \ell \leq m.
\]
Invoking the first integral $\sum_{j=0}^m c_j$ yields
\[
 c_0 = {e_0}/\left({1 + \sum\limits_{j=1}^{m}  \prod\limits_{i=1}^{\ell} \frac{k_{2i-1} s + k_{-2i} p }{k_{-(2i-1)} + k_{2i} }}\right)
\]
whence $c_0$ and all $c_j$ are of order $e_0$. 
As an intermediate result one finds
\[
\begin{array}{rcl}
 \dot s &=& - c_0 \sum\limits_{j=0}^{m-1} \frac{k_{-(2j+1)}k_{2j+1}s - k_{-(2j+1)}^2 p }{k_{-(2j+1)} + k_{2(j+1)}} \prod\limits_{i=1}^{j} \frac{k_{2i-1} s + k_{-2i} p 
}{k_{-(2i-1)} + k_{2i} }\\
 &=&\widetilde N/\widetilde D
\end{array}
\]
with
\[
\begin{array}{rcl}
\widetilde N &=& {- e_0 \sum\limits_{j=0}^{m-1} \frac{k_{-(2j+1)}k_{2j+1}s -k_{-(2j+1)}^2 p  }{k_{-(2j+1)} + k_{2(j+1)}} \prod\limits_{i=1}^{j} \frac{k_{2i-1} s + k_{-2i} p 
}{k_{-(2j-1)} + k_{2j} }}\\
\widetilde D&=&{1 + \sum\limits_{j=1}^{m}  \prod\limits_{i=1}^{\ell} \frac{k_{2i-1} s + k_{-2i} p }{k_{-(2i-1)} + k_{2i} }}.
\end{array}
\]
Using the first integral involving $p$, one sees that $p=s_0-s+e_0(\cdots)$, hence for first order in $e_0$ one obtains the reduced one-dimensional equation for a cooperative system with $m$ complexes:
\[
\dot s=N/D
\]
with 
\[
\begin{array}{rcl}
 N &=& {- e_0 \sum\limits_{j=0}^{m-1} \frac{k_{-(2j+1)}k_{2j+1}s -k_{-(2j+1)}^2 (s_0-s)  }{k_{-(2j+1)} + k_{2(j+1)}} \prod\limits_{i=1}^{j} \frac{k_{2i-1} s + k_{-2i} (s_0-s)
}{k_{-(2j-1)} + k_{2j} }}\\
 D&=&{1 + \sum\limits_{j=1}^{m}  \prod\limits_{i=1}^{\ell} \frac{k_{2i-1} s + k_{-2i} (s_0-s) }{k_{-(2i-1)} + k_{2i} }}.
\end{array}
\]
Note that the right-hand side is a rational function of $s$, with numerator and denominator of degree $m$.
\subsection{Cooperativity with two complexes }
We now consider the cooperative system with $m=2$ in greater detail; we are interested in QSS parameter values for all possible combinations of complexes. Using the two linear first integrals one has a three-dimensional system for $s$, $c_1$ and $c_2$.
We will not discuss all possible varieties and reductions, but just provide an overview of results.
\begin{itemize}
\item QSS parameter values for $c_1$ and $c_2$. Computing the elimination ideal according to Proposition \ref{adhoccrit} (with standard software) yields two generators
\[
k_3k_1^2e_0^2k_2^2(k_{-3}+k_4)^2 \text{  and  } k_1^2e_0^2k_2^2(k_{-3}+k_4)^2(k_2+k_{-1}).
\]
Thus one obtains the following four QSS-critical parameter values:
\[
k_1=0;\quad e_0=0;\quad k_2=0;\quad k_{-3}=k_4=0.
\]
According to \cite{godiss}, Kap.~9.4 (where a case-by-case discussion is given) all of these are TF-critical.
\item QSS parameter values for $c_2$. Here the ideal $J$ (see Proposition \ref{adhoccrit}) admits a Groebner basis with six generators, but
the straightforward computation of the elimination ideal with standard software is not feasible. On the other hand, the QSS parameter values for an affine coordinate subspace (according to Appendix, Remark \ref{easyrem}) can be determined: One obtains only $k_3 =0 $ (all other parameters $>0$) as QSS-critical parameter value, and the rank condition on $D_2h^{[2]}$ (see Definition \ref{qssbase}) is satisfied. (This parameter value is not TF-critical).
The QSS variety is given by \(c_2=0\), and the QSS reduced system 
\[
 \begin{array}{rcl}
  \dot s &=& -k_1 e_0 s + (k_{-1}+k_1 s) c_1 - k_3 sc_1\\
  \dot c_1 &=& k_1 e_0 s - (k_{-1}+k_2+k_1 s) c_1 + k_3 s c_1 
 \end{array}
\]
corresponds to the Michaelis-Menten system for one complex.\\
Note that the QSS conditions for $c_2$ alone are disjoint from those characterizing QSS for both complexes.
\item QSS parameter values for $c_1$. The ideal $J$ admits a Groebner basis with six generators, and the elimination ideal $J\cap\mathbb R[\pi]$ turns out to be trivial. The computation of further elimination ideals such as $J\cap\mathbb R[\pi,s]$ is not feasible with standard software, hence a complete picture is unavailable. But standard methods suffice to determine the QSS-critical parameter values for affine coordinate subspaces. One obtains two of these, viz.  $k_1 =0 $ and $k_{-3}=k_4=0$. Both already occurred in the discussion of QSS for both complexes.

\end{itemize}

\subsection{A model for decomposition of propanone}
This example is a modification of the one in Pantea et al. \cite{pgrc}, subsection 2.3, which describes the photochemical decomposition of propanone. Here we illustrate how a QSS reduction is effectively computed via singular perturbations, using Proposition \ref{allthesame}. The differential equation system we consider is as follows:
\[
\begin{array}{rcl}
 \dot c_A &=& -k_1c_A + k_{-1}c_Xc_Y-k_3c_Ac_Y \\
 \dot c_B &=& k_2c_X \\
 \dot c_C &=& k_4c_Y^2 \\
 \dot c_D &=& k_3c_Ac_Y \\
 \dot c_E &=& k_5c_Yc_Z \\
 \dot c_F &=& k_6c_Z^2 \\
 \dot c_G &=& k_7c_Z \\
 \dot c_H &=& k_8c_X^2\\
 \dot c_X &=& k_1c_A + k_{-1}c_Xc_Y-k_2c_X - 2k_8c_X^2 \\
 \dot c_Y &=& k_1c_A + k_{2}c_X - k_{-1}c_Xc_Y-k_3c_Ac_Y-2k_4c_Y^2 - k_5 c_Yc_Z + k_7c_Z \\
 \dot c_Z &=& k_3c_Ac_Y - k_5 c_Yc_Z - 2k_6c_Z^2-(k_7+k_9)c_Z
\end{array}
\]
Our modification consists of including the additional parameter $k_9$; the interpretation of this would be additional degradation of $Z$. (Admittedly, we introduce this additional parameter for technical reasons; see below.)
The interest here lies in QSS with respect to $(c_X,\,c_Y,\,c_Z)$. As proven in \cite{pgrc}, the resulting algebraic equations are generally not solvable by radicals (this fact is unaffected by the introduction of $k_9$).\\
An attempt to obtain all QSS parameter values via Proposition \ref{adhoccrit}, using standard software, works only partially: One finds a Groebner basis for the ideal $J$ but the elimination ideal is beyond reach. In view of subsection \ref{subsqssspt} and Remark \ref{easyrem} in the Appendix we are again content to find those QSS parameter values which correspond to affine coordinate subspaces. For these one obtains the conditions 
\[
\begin{array}{rcccl}
k_1&=& k_{-1}c_Xc_Y-k_2c_X - 2k_8c_X^2&=&0 \\
k_1-k_3c_Y&=& k_{2}c_X - k_{-1}c_Xc_Y-2k_4c_Y^2 - k_5 c_Yc_Z + k_7c_Z&=&0 \\
 k_3c_Y &=&- k_5 c_Yc_Z - 2k_6c_Z^2-(k_7+k_9)c_Z&=&0
\end{array}
\]
by comparing coefficients of powers of $c_A$. The result (most easily obtained via using nonnegativity of parameters and variables) is that only $k_1=0$ (all other parameters $>0$) defines a QSS parameter value, with the QSS variety $S$ defined by
$c_X= c_Y= c_Z = 0$. (There exist other QSS-critical parameter values but these do not satisfy the rank condition on $D_2h^{[2]}$.) We now use Proposition \ref{allthesame} with
\[
\widetilde A  := \begin{pmatrix} -k_2 & 0 & 0 \\ k_2 & -k_3c_A & k_7 \\ 0 & k_3 c_A & - (k_7+k_9) \end{pmatrix}, \qquad \widetilde B:= \begin{pmatrix} 0 & -k_3 c_A & 0 \\ k_2 & 0 & 0 \\ 0 & 0& 0 \\ 0 & k_3 c_A & 0 \\ 0 & 0& 0 \\ 0 & 0& 0 \\0 & 0& k_7 \\0 & 0& 0 \\\\\end{pmatrix},
\]

\[
 u := \begin{pmatrix} -c_A \\ 0 \\ 0 \\ 0 \\ 0 \\ 0\\0\\0\end{pmatrix}, \qquad v := \begin{pmatrix} c_A \\ c_A\\ 0\end{pmatrix}.
\]
(The entries depend, in principle, on $c_A$ through $c_H$, but in this special system only $c_A$ actually occurs. The notation $\widetilde A$ etc. was introduced here to distinguish matrices from chemical species.)
The reduced system is given by
\[
\begin{array}{rcl}
 \frac{d}{dt} \begin{pmatrix} c_A \\ c_B \\ c_C \\ c_D \\ c_E \\ c_F \\ c_G\\ c_H \end{pmatrix} &=& k_1 \cdot \left(u - \widetilde B \widetilde A^{-1}v \right)\\
&=& \begin{pmatrix} -k_1 c_A\cdot(3+2k_7/k_9)\\ k_1 c_A\\ 0 \\2k_1c_A(\cdot 1+k_7/k_9) \\ 0 \\ 0 \\ k_1 c_A\cdot k_7/k_9\\ 0
    \end{pmatrix}

\end{array}
\]
which (for this special system) boils down to the elementary one-dimensional equation
\[
\dot c_A= -k_1 c_A\cdot(3+2k_7/k_9)
\]
and simple quadratures. Thus, while it is an undeniable fact that the ``exact" resolution of the QSS conditions cannot be obtained by radicals, the lowest order approximation can be determined and discussed with little effort. \\
As mentioned above, we changed the system in \cite{pgrc} by introducing an extra parameter $k_9$. The technical reason for this is to ensure applicability of standard singular perturbation reduction. The original system corresponds to $k_9=0$. In this case  $\widetilde A$ is not invertible, and the scenario with  $k_1=0$ and QSS variety given by $c_X=c_Y=c_Z=0$ is singular beyond the reach of standard singular perturbation theory.

\section{Appendix}
For the reader's convenience we collect here some (known) facts from various disciplines, some technical proofs, as well as supplementary material and examples.
\subsection{Some facts about algebraic varieties}\label{algvarapp}
We collect some properties of real and complex algebraic varieties; proofs and details can be found in Kunz \cite{kunz} and Shafarevich \cite{schafa} (in particular Ch.~2, \S2-3). Let $\mathbb K$ stand for $\mathbb R$ or $\mathbb C$.
\begin{itemize}
\item We call a subset $Y$ of $\mathbb K^n$ {\em Zariski closed} if it is the common zero set of a collection $(\phi_i)_{i\in I}$ of polynomials. Conversely, given any subset $M\subseteq \mathbb K^n$, its vanishing ideal 
\[
J(M)=\left\{\psi; \psi \text{  polynomial and  }\psi(M)=0\right\}
\]
is a radical ideal in $\mathbb K[x_1,\ldots, x_n]$. The zero set of $J(M)$ is called the Zariski closure of $M$. A subset of $\mathbb K^n$ is called Zariski open if its complement is Zariski closed. The Zariski open sets form a topology on $\mathbb K^n$.
\item A Zariski closed $Y\subseteq\mathbb K^n$ is called {\em reducible} if it is the union of two proper Zariski closed subsets, and {\em irreducible} otherwise. Any Zariski closed set is a union of finitely many irreducible ones, which are called its irreducible components.
\item For the purpose of this paper, a {\em subvariety}  $V\subseteq \mathbb K^n$ (briefly, a  variety) is a relatively Zariski open subset of a Zariski closed $Y\subseteq \mathbb K^n$. We call $V$ irreducible if its Zariski closure has this property. 
\item The {\em tangent space} to $V$ at $y\in V$ is the intersection of the kernels of all $D\phi(y)$, with $\phi\in J(V)$. 
\item We call a point $y$ of a variety $V$ {\em simple} if (i) $y$ is contained in just one irreducible component $W$ of $V$, and (ii) the tangent space to $W$ at $y$ has minimal dimension. The simple points of an irreducible variety $W$ form a submanifold of $\mathbb K^n$, and its dimension is equal to the dimension of the tangent space at any simple point.
\item If $V$ is an irreducible $r$-dimensional subvariety of $\mathbb K^n$ and $y\in V$ is a simple point then (with regard to the Zariski topology) a relatively open neighborhood of $y$ in $V$ can be represented as the common zero set of $n-r$ polynomials in $J(V)$.
\end{itemize}

\subsection{Invariance and invariance criteria}\label{invapp}
We consider an ordinary differential equation
\begin{equation}\label{genericdeq}
\dot x=f(x)
\end{equation}
on a nonempty open subset $U\subseteq \mathbb R^n$, with $f:\,U\to\mathbb R^n$ smooth.
Given an open subset $\widetilde U$ of $\mathbb R^n$ and a smooth function $\theta:\, \widetilde U\to \mathbb R$, the {\em Lie derivative} of $\theta$ with respect to $f$ is defined by
$L_f(\theta)(x)=D\theta(x)f(x)$. The Lie derivative describes the rate of change for $\theta$ along solutions of \eqref{genericdeq}; it is therefore relevant for invariance criteria.
\begin{lemma}\label{invlem}
\begin{enumerate}[(a)]
\item Let $\theta_1,\ldots,\theta_s$ be smooth $\mathbb R$-valued functions on $\widetilde U\subseteq U$, and assume that there are smooth functions $\rho_{jk}$ on $\widetilde U$ such that
\begin{equation}\label{invcrit}
L_f(\theta_j)=\sum_{k=1}^s\rho_{jk}\theta_k,\quad 1\leq j\leq s.
\end{equation}
Then the common zero set $Y$ of the $\theta_j$ is an invariant set of \eqref{genericdeq}; i.e., for all $y\in Y$ the solution trajectory through $y$ is contained in $Y$.
\item Conversely, if  $Y$ is invariant then every $L_f(\theta_j)$ vanishes on the common zero set of $\theta_1,\ldots,\theta_s$.
\item A stronger converse of part (a) holds near any point $y\in Y$ at which the Jacobian of $(\theta_1,\ldots, \theta_s)$ has rank $s$: Invariance of the set $Y$ implies a relation \eqref{invcrit} in some neighborhood of $y$, with smooth functions $\rho_{jk}$.
\item For polynomial or rational functions $\theta_i$ and vector fields $f$, given the full rank condition for the Jacobian of $(\theta_1,\ldots, \theta_s)$ at $y$, invariance of the set $Y$ will imply a relation \eqref{invcrit} with rational functions $\rho_{jk}$ that are regular in $y$.
\end{enumerate}
\end{lemma}
\begin{proof} The statement of part (a) is common knowledge; see for instance \cite{CLPW}, Lemma 2.1. To prove part (b) and (c), note that (local) invariance forces $L_f(\theta_j)=0$ on the common zero set of $\theta_1,\ldots,\theta_s$, and that in the full rank case (due to a theorem by Hadamard) every function which vanishes on this zero set is locally a linear combination of the $\theta_i$ with smooth cofficients. For part (d) the argument in \cite{CLPW} works in principle, with some modification: In the complexification, consider the local ring of $y$. By Shafarevich \cite{schafa}, Ch.~2, \S3, Thms.~4 and 5 (see also Ch.~2, \S2) the functions $\theta_1,\ldots,\theta_s$ generate the vanishing ideal of $Y$ in this local ring. Due to invariance, all $L_h(\theta_j)$ are elements of this vanishing ideal, and the assertion follows for the complex case. Taking real parts, one is done.
\end{proof}
\subsection{Dependency results}\label{depresults}
In this subsection we consider smooth differential equations
\[
\dot x=f(x) \text{  and  }\dot x = g(x) \text{  on  }U.
\]
The proof of Proposition \ref{consapprox}(b) readily follows from the arguments below with $f(x)=h(x,\pi)$ and $g(x)=h_{\rm red}(x,\pi)$ with $\pi$ fixed. (Note that the assumptions for Proposition \ref{consapprox} hold uniformly in some parameter range.) Although elementary, this fact seems to be less familiar; therefore we prove it in detail here.\\
For $y\in U$ denote by $F(t,y)$ (resp. $G(t,y)$) the solution of the initial value problem $\dot x=f(x),\,x(0)=y$ (resp. $\dot x=g(x),\,x(0)=y$).\\
We will always consider the maximum norm $\Vert\cdot\Vert=\Vert\cdot\Vert_\infty$ on $\mathbb R^n$ and its associated operator norm. With $\overline{B_r(y)}$ we denote the closed ball with center $y$ and radius $r$. Moreover we let $K\subseteq U$ be compact with nonempty interior, and let $R>0$ such that $\Vert f(x)\Vert\leq R$ and $\Vert g(x)\Vert\leq R$ for all $x\in K$. We note a basic result first.
\begin{lemma}\label{escape}
Let $y\in K$ and $r>0$ such that  $\overline{B_r(y)}\subseteq K$. Then $F(t,y)$ and $G(t,y)$ exist and are contained in $\overline{B_r(y)}$ for all $t\in \left[0,\,r/R\right]$.
\end{lemma}
\begin{proof}
Abbreviate $z(t)=F(t,y)$. Then for $t>0$
\[
\begin{array}{rcccl}
z(t)-y&=&\int_0^t\dot z(s)\,ds&=&\int_0^t f(z(s))\, ds,\\
\text{so  }\Vert z(t)-y\Vert&\leq&\int_0^t\Vert f(z(s))\Vert\,ds&\leq &R\cdot t.\\
\end{array}
\]
Existence follows from the fact that the trajectories are contained in a compact subset of $\mathbb R^n$.
\end{proof}
Next we obtain a lower estimate for the norm of the difference of solutions when $f(y)\not=g(y)$ at some initial value $y$. The ingredients in the proof are standard.
\begin{lemma}\label{nonclose}
In addition to the above, let $\Vert Df(x)\Vert \leq L$ and $\Vert Dg(x)\Vert \leq L$ for all $x\in K$, with some $L>0$. Let $y$ be an interior point of $K$, and assume that $\Vert f(y)-g(y)\Vert =|f_i(y)-g_i(y)|\geq2\rho$, with suitable $\rho>0$ and $i, \,1\leq i\leq n$.
\begin{enumerate}[(a)]
\item Let $d>0$ such that $\overline{B_d(y)}\subseteq K$ and $|f_i(x_1)-g_i(x_2)| \geq \rho$ for all $x_1, x_2\in \overline{B_d(y)}$. Then for $t_1:=d/R$ one has that $F(t,y)$ and $G(t,y)$ are contained in $K$ for all $t\in \left[0,\,t_1\right]$ and $\Vert F(t_1,y)-G(t_1,y)\Vert \geq \rho d/R$.
\item If $\overline{B_{\rho/2L}(y)}\subseteq K$ then for $t_2:=\rho/(2LR)$ one has that  $F(t,y)$ and $G(t,y)$ are contained in $K$ for all $t\in \left[0,\,t_2\right]$ and $\Vert F(t_2,y)-G(t_2,y)\Vert \geq \rho^2/2RL$.
\end{enumerate}
\end{lemma}
\begin{proof} (a) We may assume that $f_i(y)>g_i(y)$. Abbreviate $z(t):=F(t,y)$ and $w(t):=G(t,y)$. Then for $t>0$
\[
\begin{array}{rcl}
z_i(t)-w_i(t)&=& \int_0^t f_i(z(s))-g_i(w(s))\, ds,\\
\text{so  }z_i(t)-w_i(t)&\geq & t\cdot \rho\\
\end{array}
\]
as long as $z(t),\,w(t)\in\overline{B_d(y)}$. Now Lemma \ref{escape} shows the assertion.\\
(b) In view of part (a) we just need to show that $|f_i(x_1)-g_i(x_2)|\geq \rho$ for all $x_1,\,x_2\in \overline{B_d(y)}$, with $d=\rho/2L$. Define $H(x_1,x_2):=f(x_1)-g(x_2)$, thus
\[
DH(x_1,x_2)\,(v_1,v_2)=Df(x_1)v_1+Dg(x_2)v_2.
\]
With 
\[
w(s):=\left(\begin{array}{c} y\\y\end{array}\right)+s\cdot\left(\begin{array}{c} x_1-y\\x_2-y\end{array}\right),\quad 0\leq s\leq 1
\]
one obtains
\[
\begin{array}{l}
H(x_1,x_2)-H(y,y)=\int_0^1 \frac{d}{ds}\,H(w(s))\,ds\\
  =\int_0^1 Df(y+s(x_1-y))\cdot(x_1-y)+ Dg(y+s(x_2-y))\cdot(x_2-y)\, ds
\end{array}
\]
which implies
\[
\Vert H(x_1,x_2)-H(y,y)\Vert \leq L\cdot(\Vert x_1-y\Vert + \Vert x_2-y\Vert).
\]
For $x_1,\,x_2\in \overline{B_d(y)}$ with $d=\rho/2L$ one finally has
\[
|H_i(x_1,x_2)|\geq |H_i(y,y)|-\Vert H(y,y)-H(x_1,x_2)\Vert\geq 2\rho-\rho
\]
as desired.
\end{proof}
\subsection{Singular perturbation reduction}\label{sptredu}

The search for QSS-critical parameter values may lead to TF-critical parameter values, which in turn may lead to Tikhonov-Fenichel reduction in some applications. 
We recall some notions and results from \cite{gwz} and \cite{gw}; see these sources for details.
\begin{itemize}
\item A parameter value \(\widehat\pi\in \Pi\) is called a {\em Tikhonov-Fenichel parameter value (TFPV) for dimension \(s\)} ( \(1\leq s\leq n-1\))  of  system \eqref{sys} whenever the following hold:
\begin{enumerate}[(i)]
 \item The zero set $ \mathcal{V}(h(\cdot, \widehat\pi))$ of \(x\mapsto h(x\,,\widehat\pi)\) contains a local submanifold \(\widetilde V\) of dimension \(s\).
 \item There is a point \(x_0\in \widetilde V\)  such that $Dh(x,\widehat\pi)$ has rank $n-s$ and
\[
 \mathbb R^n = {\rm Ker}\ Dh(x,\widehat\pi) \oplus {\rm Im}\ Dh(x,\widehat\pi)
\]
for all $x\in \widetilde V$ near $x_0$.
\item The nonzero eigenvalues of $\ Dh(x_0,\widehat\pi) $ have real part $<0$.
\end{enumerate}
Note that condition (i) alone characterizes TF-critical parameter values; cf.~Definition \ref{tfcritdef}.
\item Given a TFPV $\widehat\pi$ and some (suitable) $\rho\in\mathbb R^m$, one obtains reduction by Tikhonov's theorem for the system
\begin{equation}\label{syseps}
\dot x = h(x,\widehat\pi+\epsilon\rho) = h(x,\widehat\pi)+\epsilon q(x)+\cdots,  \mbox{     as   }\epsilon\to 0.
\end{equation}
\item The reduced system corresponding to \eqref{syseps} is defined on the invariant manifold $\widetilde V$. To find it explicitly, one uses a decomposition
\[
h(x,\,\widehat\pi)=P(x,\,\widehat\pi)\,\mu(x,\,\widehat\pi)
\]
in some neighborhood of $x_0$. Here $P$ is an $\mathbb R^{n\times (n-s)}$--valued function of rank $n-s$ on $\widetilde V$, and $\widetilde V$ equals the vanishing set of the $\mathbb R^{(n-s)}$--valued function $\mu$. One verifies that $A(x,\,\widehat\pi):=D\mu(x,\,\widehat\pi)\,P(x,\,\widehat\pi)$ is invertible on $\widetilde V$. The reduced system on $\widetilde V$ is given by
\begin{equation}\label{redform}
\dot x=\epsilon\cdot\left(I_n-P(x,\,\widehat\pi)A(x,\,\widehat\pi)^{-1}D\mu(x,\,\widehat\pi)\right)q(x),
\end{equation}
in fast time scale resp. by
\begin{equation}\label{redformslow}
x^\prime=\left(I_n-P(x,\,\widehat\pi)A(x,\,\widehat\pi)^{-1}D_1\mu(x,\,\widehat\pi)\right)q(x)
\end{equation}
in slow time scale.
\end{itemize}
\subsection{Supplementary material}
\subsubsection{Variations of the reduced system}\label{varredsys}
 With regard to system \eqref{qssredsys}, one is only interested in its restriction to $U_\pi$.
 More generally one may therefore call any equation of the form
\begin{equation}\label{qssredsysplus}
\begin{array}{rcrcl}
\dot x^{[1]}&=& h^{[1]}(x,\pi)&+&\sum_{j\geq r+1}m^{[1]}_jh_j\\
\dot x^{[2]}&=& -D_2h^{[2]}(x,\pi)^{-1}D_1h^{[2]}(x,\pi)h^{[1]}(x,\pi)&+&\sum_{j\geq r+1}m^{[2]}_jh_j\\
\end{array}
\end{equation}
with (e.g.) rational functions $m^{[1]}_j$ and  $m^{[2]}_j$ a reduced system corresponding to \eqref{sys}, since the right hand sides of \eqref{qssredsys} and \eqref{qssredsysplus} are equal on $U_\pi$. In this respect, the reduced system is not unique.

\begin{example} Consider the irreversible Michaelis-Menten system \eqref{mimeirrev}, with QSS for complex.
The reduced system according to \eqref{qssredsys} is then given by
\[
\begin{array}{rcl}
\dot{s}&=-&k_1e_0s+(k_1s+k_{-1})c,\\
\dot{c}&= &\frac{k_1(e_0-c)}{k_1s+k_{-1}+k_2}\left(k_1e_0s-(k_1s+k_{-1})c\right).
\end{array}
\]
Using $L_h(c)=-k_1e_0s+(k_1s+k_{-1}+k_2)c=0$ one may use \eqref{qssredsysplus} to replace this system by
\[
\begin{array}{rcl}
\dot{s}&=-&k_1e_0s+(k_1s+k_{-1})c,\\
\dot{c}&= -&\frac{k_1(e_0-c)}{k_1s+k_{-1}+k_2}\cdot k_2c.
\end{array}
\]
\end{example}
(Of course, in the present example one may readily solve $L_h(c)=0$ for $c$ as a function of $s$ and obtain the familiar reduction.)
\subsubsection{Algorithmic considerations}\label{algoapp}
As noted in Subsection 3.4 above, the mathematically adequate approach for finding QSS-critical parameter values is to consider the ideal $J\subseteq \mathbb R[x,\,\pi]$ generated by the polynomials given in Proposition \ref{adhoccrit}, rather than the polynomials by themselves. A pertinent observation is the following.
\begin{remark}
If $(y^*,\pi^*)\in J$ then $\pi^*$ is a zero of the elimination ideal $J\cap \mathbb R[\pi]$. Thus the zeros of  $J\cap \mathbb R[\pi]$ are just the QSS-critical parameter values.  If $\widehat \pi$ is QSS-critical  then $\widehat\pi$ is a QSS parameter value if and only if there exists an $\widehat y\in Y_{\widehat\pi}$ such that $D_2h^{[2]}(\widehat y,\,\widehat \pi)$ has rank $n-r$.
\end{remark}

For properties of elimination ideals see e.g. Cox et al. \cite{coxII}. Standard algorithms use Gr\"obner bases and are implemented in {\sc Singular} \cite{DGPS} and other software systems. For more details, as well as examples from biochemistry, see \cite{gwz, gwz2}.  \\
Concerning feasibility, a straightforward algorithmic search for QSS-critical parameter values via Proposition \ref{adhoccrit} may quickly become cumbersome, even for relatively small systems. From an algorithmic perspective, much room for improvement remains.\\
However, in one relevant special setting the situation is better: Finding QSS parameter values which admit reduction to an affine coordinate subspace is less involved:
\begin{remark}\label{easyrem}
 Invariance of $Z_{\gamma^*}$ is equivalent to
\[
h_j(x_1,\ldots, x_r,\gamma_{r+1}^*,\ldots,\gamma_n^*,\pi)=0,\quad r+1\leq j\leq n.
\]
This opens up a shortcut for computations: To determine QSS parameter values for reduction to an affine coordinate subspace, write
\[ h_j(x_1,\ldots, x_r,\gamma_{r+1},\ldots, \gamma_n,\pi)
\]
as a linear combination of monomials in $x_1,\ldots, x_r$ with coefficients in $\mathbb R[\gamma,\pi]$. Then all these coefficients must equal zero; this yields computable conditions for $\pi$ and $\gamma$.
\end{remark}
\begin{example} Consider the reversible Michaelis-Menten system \eqref{mimerev}. Assuming QSS for complex, in the coordinate subspace setting we have
\[
\begin{array}{rcl}
h_2(s,\gamma,\pi)&=&k_1e_0s-(k_1s+k_{-1}+k_2)\gamma+k_{-2}(e_0-\gamma)(s_0-s-\gamma)\\
               &=&(k_1-k_{-2})(e_0-\gamma)s+\left((k_{-1}+k_2)\gamma+k_{-2}(e_0-\gamma)(s_0-\gamma)\right)
\end{array}
\]
View this as a polynomial in $s$, of degree one. The coefficient of $s$ yields 
\[
(k_1-k_{-2})(e_0-\gamma)=0
\]
which leads to two cases. 
\begin{itemize}
\item If the rate constants $k_1$ and $k_{-2}$ are equal, and $e_0>0$, there remains only the condition
\[
-\gamma(k_{-1}+k_2)+k_{-2}(e_0-\gamma)(s_0-\gamma)=0;
\]
thus the zeros of this quadratic function of $\gamma$ will define invariant straight lines for the system. An elementary discussion shows that both zeros are $\geq 0$, but only the smaller one is $\leq e_0$ (which is required by the initial conditions). Thus we find one invariant straight line that is of interest. (This has been observed before; see Miller and Alberty \cite{MiAl}.) Note that the assumption $e_0=0$ directly implies $\gamma=0$; see the following item.
\item If $k_1\not=k_{-2}$ then $\gamma=e_0$, with remaining condition 
\[
-\gamma(k_{-1}+k_2)=0.
\]
Thus $e_0=0$ or $k_{-1}=k_2=0$; both cases correspond to TF parameter values; see \cite{gwz}.
\end{itemize}
\end{example}
\subsubsection{Slow manifold and QSS variety}\label{slowqss}
Here we show by example that condition (i) in Definition \ref{consistency} is not an automatic consequence of $\pi^*$ being both a QSS parameter value and TF-critical.

\begin{example}
Given the first order reaction network
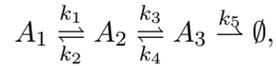
\[
 A_1 \overset{k_{1}}{\underset{k_{2}}\rightleftharpoons} A_2  
\overset{k_3}{\underset{k_{4}}\rightleftharpoons} A_3 \overset{k_{5}}{\rightharpoonup} \emptyset,
\]
the reaction equations
\[
\begin{array}{rccclcl}
 \dot x_1 &=& -k_1x_1 & + & k_2 x_2 & \\
 \dot x_2 &=& k_1x_1 &-& (k_2+k_3)x_2 & +& k_4x_3\\
 \dot x_3 &=& & & k_3 x_2 & -& (k_4+k_5) x_3\\
\end{array}
\]
admit the Tikhonov-Fenichel parameter value with $k_3=0$ and all other parameters $>0$ (differently stated, a small parameter $k_3$), with reduction to the one-dimensional slow manifold \( { S} := \{ (x_1,x_2,0)^{\rm tr} \in \R^3;\ k_1x_1 = k_2 x_2\}\). The Tikhonov-Fenichel reduction procedure described in \ref{sptredu}, with
\[
 \mu := \begin{pmatrix} -k_1 x_1 + k_2 x_2 , &  x_3\end{pmatrix} ,\  P:=\begin{pmatrix} 1 & 0 \\ -1 & k_4 \\ 0 & -(k_4+k_5) \end{pmatrix}
\]
yields a reduced system on \({ S}\), given by
\[
\dot x = -\frac{k_3 k_5 x_2}{(k_1+k_2)(k_4+k_5)} \begin{pmatrix} k_2 \\ k_1 \\0 \end{pmatrix}.
\]
On the other hand $k_3=0 $ also defines a QSS parameter value with respect to $A_3$; the QSS variety is given by $x_3 = 0$ and has dimension two. The QSS-reduced system is given by
\[
\begin{array}{rccclcl}
 \dot x_1 &=& -k_1x_1 & + & k_2 x_2 & \\
 \dot x_2 &=& k_1x_1 &-& \left(k_2 + \frac{k_3}{k_4+k_5} \right)x_2
\end{array}
\]
Thus the slow manifold is a proper subvariety of the QSS variety. One can verify that a singular perturbation reduction of the QSS-reduced system will provide the same one-dimensional equation on $S$. From a pragmatic perspective, one may prefer the direct reduction to $S$.
\end{example}
\bigskip

\noindent
{\bf Acknowledgement.} We thank two anonymous reviewers for valuable comments.

\end{document}